\documentclass[12pt, letterpaper]{article}

\usepackage[utf8]{inputenc}

\usepackage[backend=biber,
style=numeric-comp, sorting=nyt, doi=false, isbn=false, date=year]{biblatex}

\AtEveryBibitem{\clearfield{url} \clearfield{note}}

\DeclareNameAlias{author}{family-given}
\DeclareNameAlias{editor}{family-given}

\AtBeginBibliography{%
}

\usepackage[pdftex,colorlinks=true,linkcolor=blue,citecolor=blue,urlcolor=red,unicode=true,hyperfootnotes=true,bookmarksnumbered]{hyperref}
\usepackage[margin=1truein]{geometry}

\usepackage{amsfonts}

\usepackage{wrapfig}

\usepackage{caption}
\usepackage{subcaption}

\usepackage{amsfonts, amssymb, amsmath, amsthm}
\usepackage{mathtools}
\usepackage{enumerate}
\usepackage{verbatim}

\usepackage{mathrsfs,amsfonts,mathtools}
\usepackage{amsmath}
\usepackage{amssymb}

\usepackage{amsthm}
\usepackage{thmtools}

\theoremstyle{plain}
\newtheorem{theorem}{Theorem}[section]

\newtheorem{conjecture}[theorem]{Conjecture}
\newtheorem{question}[theorem]{Question}
\newtheorem{proposition}[theorem]{Proposition}
\newtheorem{corollary}[theorem]{Corollary}

\theoremstyle{definition} 

\newtheorem{definition}[theorem]{Definition}
\newtheorem*{remark*}{Remark}
\newtheorem*{definition*}{Definition}

\renewcommand{\emptyset}{\varnothing}

\renewcommand{\epsilon}{\varepsilon}
\renewcommand{\phi}{\varphi}
\renewcommand{\kappa}{\varkappa}
\newcommand{\R}{\mathbb{R}}

\newcommand{\Z}{\mathbb{Z}}
\newcommand{\Nz}{\Z_{\ge0}}

\newcommand{\br}[1]{\left({#1}\right)}
\newcommand{\fbr}[1]{\left\{#1\right\}}
\newcommand{\lbr}[1]{\left|{#1}\right|}

\newcommand{\wt}[1]{\widetilde{#1}}

\newcommand{\mc}[1]{\mathcal{#1}}
\newcommand{\mb}[1]{\mathbf{#1}}
\newcommand{\mulg}[2]{{\left[#1\right]}^{#2}}
\newcommand{\pwrs}[1]{\mc{P}\!\left(#1\right)}
\newcommand{\lcG}{G_{\mathrm{left}}}
\newcommand{\matcon}[3]{{#2}\,\xi_{#1}\,{#3}}

\newcommand{\seq}{\subseteq}

\newcommand{\shadow}[1]{\partial_{#1}}

\newcommand{\wsat}{\operatorname{\mathrm{wsat}}}
\newcommand{\rk}{\operatorname{\mathrm{rk}}}
\newcommand{\Wsat}{\mathrm{wSAT}}
\newcommand{\rhosat}{\operatorname{\rho\text{-}sat}}

\title{Asymptotically optimal lower bounds on weak saturation numbers for hypergraphs}

\author{Nikolai Terekhov\thanks{Department of Discrete Mathematics, Moscow Institute of Physics and Technology, Dolgoprudny, Russia; nikolayterek@gmail.com
}}

\date{}

\begin{document}

\maketitle

\begin{abstract}
    Given an $r$-uniform hypergraph $H$ and a positive integer $n$, the weak saturation number $\wsat(n,H)$ is the minimum number of edges in an $r$-uniform hypergraph $F$ on $n$ vertices such that the missing edges in $F$ can be added, one at a time, so that each added edge creates a copy of $H$.
    For the case of graphs ($r = 2$), asymptotically optimal general lower bounds for these numbers in terms of the minimum vertex degree of $H$ are known. In this work, we generalize these bounds to the case of hypergraphs and establish their asymptotic optimality. To prove this, we introduce a lower bound method based on polymatroids. This method generalizes a linear algebraic method but, unlike the original version, makes it possible to derive lower bounds with non-integer asymptotic coefficients.
\end{abstract}

\section{Introduction}

\begin{definition}
    Let $r \ge 1$ and $n \ge 1$ be integers, and let $H$ be a non-empty\footnote{Here and throughout, a hypergraph $H$ is called non-empty if $|E(H)|\ge 1$.} $r$-uniform hypergraph. We say that an $r$-uniform hypergraph $F$ on $n$ vertices is \emph{weakly \mbox{$H$-saturated}} if it is possible to arrange the edges $E(K_n^r) \setminus E(F)$, where $K_n^r$ denotes the complete $r$-uniform hypergraph on $n$ vertices, in an order $e_1, \ldots, e_k$ such that the addition of each $e_i$ to the hypergraph $F \cup \fbr{e_1, \ldots, e_{i-1}}$ creates a new copy of $H$. In this case, we write $F \in \Wsat(n, H)$. The minimum number of edges in such a hypergraph $F$ is denoted $\wsat(n, H)$ and is called the \emph{weak saturation number} of $H$ in $K_n^r$.
\end{definition}

The notion of weak saturation was introduced by Bollob\'{a}s~\cite{1968Bollobas}. These numbers have been studied extensively~\cite{
    1982Frankl,
    1985Alon,
    1985Kalai,
    1991Erdos,
    1992Tuza,
    2001PikhurkoExterior,
    2012BaloghLinearAlgebra,
    2015Moshkovitz,
    2018Morrison,
    2020Hambardzumyan,
    2023Bulavka,
    2023Shapira,
    2025TerekhovLinearAlgebra},
and, in particular, exact values of $\wsat(n,H)$ have been determined when $H$ is a complete $r$-uniform hypergraph~\cite{1982Frankl,1984KalaiRigid,1985Alon}, a complete bipartite graph with equal part sizes~\cite{2021Kronenberg}, or a pyramid~\cite{2001PikhurkoExterior}. Beyond computing these quantities for particular hypergraphs, general bounds have also been sought. One of the earliest results in this direction for graphs ($r=2$) is the following upper bound in terms of the minimum vertex degree $\delta(H)$ of a graph $H$.
\begin{proposition}[\cite{2007Sidorowicz}]\label{prop:trivial-upper-bound-graph}
    Let $H$ be a graph with $\delta(H)=\delta \ge 1$. Then
    \[\wsat(n,H) \le  (\delta - 1)\cdot n + O(1).\]
\end{proposition}
This upper bound is asymptotically optimal, since, for example, it is attained by the clique $K^2_{\delta+1}$~\cite{1982Frankl, 1984KalaiRigid, 1985Kalai}, for which $\wsat(n,K^2_{\delta+1})=(\delta-1)\cdot n - \binom{\delta}{2}$ for all $n\ge \delta+1$.

In view of this result, it is natural to seek a general lower bound in terms of $\delta(H)$ as well. In the same paper~\cite{2007Sidorowicz}, the following result of this kind was established.
\begin{proposition}[\cite{2007Sidorowicz}]\label{prop:trivial-lower-bound-graph}
    Let $H$ be a graph with $\delta(H)=\delta \ge 1$. Then for all \mbox{$n\ge \lbr{V(H)}$},
    \[\wsat(n,H) \ge \frac{\delta-1}{2}\cdot n.\]
\end{proposition}

This bound was later strengthened as follows~\cite{2013Faudree, 2025TerekhovComb}.
\begin{theorem}[\cite{2013Faudree, 2025TerekhovComb}]\label{thrm:optimal-lower-bound-graph}
    Let $H$ be a graph with $\delta(H)=\delta \ge 1$. Then for all \mbox{$n\ge \lbr{V(H)}$},
    \[\wsat(n,H) \ge \br{\frac{\delta}{2}-\frac{1}{\delta + 1}}\cdot n.\]
\end{theorem}

Moreover, the above bound is asymptotically optimal, as the following result shows~\cite{2025TerekhovComb}.
\begin{proposition}[\cite{2025TerekhovComb}]\label{prop:graph-lower-bound-is-optimal}
    For every integer $\delta\ge 1$ there exists a graph $H$ with $\delta(H) = \delta$ such that
    \[\wsat(n,H) \le \br{\frac{\delta}{2}-\frac{1}{\delta + 1}}\cdot n + O(1).\]
\end{proposition}

In this paper, we investigate a generalization of this asymptotically optimal lower bound to $r$-uniform hypergraphs. The role of $\delta(H)$ is played by the minimum positive codegree parameter $\delta^*(H)$, which is defined as the minimum positive degree of an $(r-1)$-subset of vertices.

\begin{definition}
    Let $r \ge 1$ be an integer. For an $r$-uniform hypergraph $H$ and a subset $U \seq V(H)$ define the link hypergraph
    \[L_H(U)=\fbr{e\setminus U\mid e\in E(H),\ U\seq e}.\]
    For a non-empty $H$, set\footnote{Here for a set $T$ and integer $k \ge 0$, $\binom{T}{k}$ denotes the set $\fbr{S \seq T \mid \lbr{S}=k}$.}
    \[\delta^*(H)=\min\fbr{\lbr{L_H(U)}\ \middle|\ \ U\in \binom{V(H)}{r-1},\ \lbr{L_H(U)}\neq 0}.\]
    For an empty $H$, set $\delta^*(H)=0$.
\end{definition}

For graphs ($r=2$) without isolated vertices, this parameter coincides with $\delta(H)$. There is a general upper bound on $\wsat(n,H)$ in terms of $\delta^*(H)$~\cite{2023Bulavka} that mirrors the corresponding upper bound for graphs.
\begin{proposition}[\cite{2023Bulavka}]\label{prop:trivial-upper-bound-hyper}
    Let $r \ge 1$ and $\delta\ge 1$ be integers. Let $H$ be an $r$-uniform hypergraph with $\delta^*(H)=\delta$. Then
    \[\wsat(n,H)\le (\delta-1)\cdot\binom{n}{r-1} + O(n^{r-2}).\]
\end{proposition}
This upper bound is asymptotically optimal, since it is attained by the complete $r$-uniform hypergraph $K^r_{r-1+\delta}$, for which $\wsat(n,K^r_{r-1+\delta})=\binom{n}{r}-\binom{n-\delta+1}{r}=(\delta-1)\cdot\binom{n}{r-1} + O(n^{r-2})$.

In this paper, we establish a lower bound on $\wsat(n,H)$ for $r$-uniform hypergraphs in terms of $\delta^*(H)$, thereby extending the graph-case bound from Theorem~\ref{thrm:optimal-lower-bound-graph}.
\begin{theorem}\label{thrm:optimal-lower-bound-hyper}
    Let $r \ge 1$ and $\delta\ge 1$ be integers. Let $H$ be an $r$-uniform hypergraph with $\delta^*(H)=\delta$. Then for all $n\ge \lbr{V(H)}$,
    \[\wsat(n,H)\ge \br{\frac{\delta}{r}-\frac{1}{\binom{r+\delta - 1}{r - 1}}}\cdot\binom{n}{r-1}.\]
\end{theorem}

We also show that this bound is asymptotically optimal.
\begin{theorem}\label{thrm:hyper-lower-bound-is-optimal}
    For every integers $r \ge 1$ and $\delta\ge 1$ there exists an $r$-uniform hypergraph $H$ with $\delta^*(H)=\delta$ such that
    \[\wsat(n,H) \le \br{\frac{\delta}{r}-\frac{1}{\binom{r+\delta - 1}{r - 1}}}\binom{n}{r-1} + o(n^{r-1}).\]
\end{theorem}

\paragraph{Proof Strategy}

We prove Theorem~\ref{thrm:optimal-lower-bound-hyper} using a modification of a linear algebraic method. In this method, in order to derive a lower bound on $\wsat(n,H)$ for an $r$-uniform hypergraph $H$, one considers a matroid $M$ on the edge set $E(K_n^r)$ of the complete $r$-uniform hypergraph. If $M$ satisfies certain compatibility conditions with $H$ --- namely, that for every copy $\wt{H}$ of $H$, the set $E(\wt{H})$ is a union of circuits of $M$ --- then one obtains the lower bound $\wsat(n,H)\ge \rk(M)$.

This method was introduced in full generality by Kalai~\cite{1985Kalai} and has been used to determine $\wsat(n,H)$ for many hypergraphs $H$~\cite{2001PikhurkoExterior, 2021Kronenberg}, including complete $r$-uniform hypergraphs~\cite{1982Frankl,1985Alon, 1984KalaiRigid}. However, in~\cite{2025TerekhovLinearAlgebra} it was shown that, in the case of graphs ($r=2$), the best lower bound obtained using Kalai's linear algebraic method always has an integer asymptotic coefficient; that is, it is of the form $a\cdot n (1+o(1))$, where $a$ is an integer. This imposes limitations on the scope of the method.

In~\cite{2025TerekhovLinearAlgebra}, Kalai's linear algebraic method was modified using multigraphs, which made it possible to avoid the issue of integer coefficients. Inspired by this modification, in this work we show how to modify the linear algebraic method using polymatroids~\cite{1978Fujishige,2003Edmonds}, which enables us to obtain bounds with non-integer coefficients.

To apply this method in Theorem~\ref{thrm:optimal-lower-bound-hyper}, we require a family of polymatroids that is amenable to compatibility conditions with a pattern hypergraph $H$, analogous to those arising in the matroid setting. For this purpose, we construct a family of count polymatroids generalizing count matroids from Pikhurko's work~\cite{2001PikhurkoCount}. These polymatroids are combinatorial in nature and are therefore readily adapted to satisfy compatibility conditions of a given hypergraph. Using this family of polymatroids together with the modified linear algebraic method, we prove a more general bound, namely Theorem~\ref{thrm:gamma-lower-bound-hyper}, from which Theorem~\ref{thrm:optimal-lower-bound-hyper} follows.

As for Theorem~\ref{thrm:hyper-lower-bound-is-optimal}, we prove it by providing, for each $r\ge 1$ and $\delta \ge 1$, an explicit construction of a hypergraph $H$. To obtain the desired upper bound on $\wsat(n,H)$, we follow a strategy similar to that of~\cite{2023Shapira, 2025TerekhovTuza}: R\"odl's theorem~\cite{1985Rodl} is used to combine upper bounds on $\wsat(m,H)$ for small values of $m$ into an upper bound on $\wsat(n,H)$ for large $n$.

\paragraph{Paper Structure} In Section~\ref{sec:preliminary-wsat} we present preliminary facts about weak saturation.
In Section~\ref{sec:linear-algebra-method} we describe our modification of Kalai's linear algebraic method using polymatroids.
In Section~\ref{sec:count-polymatroids} we discuss count matroids and construct a family of count polymatroids.
In Section~\ref{sec:gamma-lower-bound-hyper} we use count polymatroids in the modified linear algebraic method and prove a general lower bound of $\wsat(n,H)$ in terms of some combinatorial characteristic of hypergraph $H$ (Theorem~\ref{thrm:gamma-lower-bound-hyper}).  We also discuss its connection with a related bound in the graph case obtained in~\cite{2025TerekhovComb} (see Subsection~\ref{subsec:gamma-connection-with-graph-case}).
In Section~\ref{sec:delta-lower-bound-hyper} we prove a lower bound on the combinatorial characteristic appearing in the general lower bound from Section~\ref{sec:gamma-lower-bound-hyper}, which in turn yields Theorem~\ref{thrm:optimal-lower-bound-hyper}.
In Section~\ref{sec:optimality} we construct hypergraphs $H$ with prescribed upper bounds on $\wsat(n,H)$ and prove Theorem~\ref{thrm:hyper-lower-bound-is-optimal}.
In Section~\ref{sec:arbitrary-sharpness-bounds} we discuss possible generalizations of Theorem~\ref{thrm:optimal-lower-bound-hyper}. In Section~\ref{sec:polymatroid-asympt} we discuss the best bound obtainable by the polymatroidal method. 

\paragraph{Notation} For an integer $n \ge 1$, let $[n]$ denote the set $\fbr{1,2,\ldots,n}$. For a set $T$ and integer $k \ge 0$, let $\binom{T}{k}$ denote the set $\fbr{S \seq T \mid \lbr{S}=k}$. For integers $n \ge 1$ and $r \ge 1$, let $K_n^r$ denote the complete $r$-uniform hypergraph on $n$ vertices. If a hypergraph $G$ is defined by specifying only its edge set $E$, then the vertex set is understood to be $V(G)=\cup_{e\in E}\ e$.


\section{Preliminary facts about weak saturation}\label{sec:preliminary-wsat}

\subsection{Asymptotic behavior of weak saturation}

In the case $\delta^*(H) = 1$, the bound in Theorem~\ref{thrm:optimal-lower-bound-hyper} is trivial. Therefore, for greater generality, we formulate general lower bounds in terms of \emph{sparseness}~\cite{1992Tuza, 2023Shapira}, which characterizes the asymptotic behavior of $\wsat(n,H)$.

\begin{definition}\label{def:sparseness}
    Let $r \ge 1$ be an integer and let $H$ be an $r$-uniform hypergraph.

    For a non-empty $H$, define the \emph{sparseness} $s(H)$ as the size of the smallest subset $S \subseteq V(H)$ such that there exists exactly one edge $U \in E(H)$ with $S \subseteq U$.

    For a empty $H$, set $s(H)=-1$.
\end{definition}

\begin{theorem}[\cite{1992Tuza}]\label{thrm:tuza-theta-wsat}
    Let $r \ge 1$ and $s\ge 1$ be integers. Let $H$ be an $r$-uniform hypergraph with $s(H) = s$. Then
    \[
        \wsat(n, H) = \Theta(n^{s - 1}).
    \]
\end{theorem}
Moreover, the weak saturation has an asymptotic coefficient, as the following result shows.
\begin{theorem}[\cite{2023Shapira}]\label{thrm:tuza-limit-wsat}
    Let $r \ge 1$ and $s\ge 1$ be integers. Let $H$ be an $r$-uniform hypergraph with $s(H) = s$. Then the following limit exists:
    \[
        \lim_{n\to+\infty}\frac{\wsat(n, H)}{n^{s-1}}.
    \]
\end{theorem}

An important property of sparseness is given by the following statement from~\cite{1992Tuza}, which helps in constructing weakly $H$-saturated hypergraphs.

\begin{proposition}[\cite{1992Tuza}]\label{prop:short-tuza-final-building-step}
    Let $r \ge 1$ and $s \ge 1$ be integers. Let $H$ be an $r$-uniform hypergraph with $s(H) = s$. Let $F$ be an $r$-uniform hypergraph with a subset $Z \subseteq V(F)$ such that $|Z| \ge |V(H)| - s$ and $F$ contains all edges $e$ such that $|e \setminus Z| \le s - 1$. Then $F$ is weakly $H$-saturated.
\end{proposition}

\begin{proof}
    For completeness, we provide the proof.

    We add the missing edges into $F$ in order of increasing $|e \setminus Z|$.

    Suppose we are currently adding an edge $e \in \binom{V(F)}{r}$. Let $\wt{F}$ denote the hypergraph $F$ with the already added edges. By the definition of sparseness, there exists a set $S \in \binom{V(H)}{s}$ such that there is a unique edge $U \in E(H)$ with $S \subseteq U$. Take in $\wt{F} \cup \{e\}$ a copy $\wt{H}$ of the hypergraph $H$ such that $V(\wt{H}) \setminus e$ is a subset of $Z$, the edge $e$ is the preimage of edge $U$, and the preimage $\wt{S}$ of the set $S$ is a subset of $e \setminus Z$. This is possible since $|e \setminus Z| \ge s$. Then, for all edges $W \in E(\wt{H}) \setminus \{e\}$, we have $|W \cap \wt{S}| < |\wt{S}|$, hence $|W \setminus Z| < |e \setminus Z|$. Thus, $W$ indeed belongs to $E(\wt{F})$, and the edge $e$ creates a new copy of the hypergraph $H$.
\end{proof}

\subsection{Weak saturation for family of hypergraphs}

The definition of weak saturation naturally generalizes to the case of a family of hypergraphs.

\begin{definition}
    Let $r \ge 1$ be an integer and let $\mc{H}$ be a non-empty family of non-empty $r$-uniform hypergraphs. We say that an $r$-uniform hypergraph $F$ on $n$ vertices is \emph{weakly \mbox{$\mc{H}$-saturated}} if it is possible to arrange the edges $E(K_n^r) \setminus E(F)$ in an order $e_1, \ldots, e_k$ such that for every $i$ there exists a hypergraph $H \in \mc{H}$ for which adding $e_i$ to the hypergraph $F \cup \{e_1, \ldots, e_{i-1}\}$ creates a new copy of $H$. In this case, we write $F \in \Wsat(n, \mc{H})$. The minimum number of edges in such a hypergraph $F$ is denoted $\wsat(n, \mc{H})$ and is called the \emph{weak saturation number} of $\mc{H}$ in $K_n^r$
\end{definition}

Weak saturation numbers for families of hypergraphs are closely connected to weak saturation numbers for the \emph{disjoint union} of hypergraphs.
\begin{definition}
    Let $r \ge 1$ and $m \ge 1$ be integers, and let $\mc{H} = \{H_i\}_{i \in [m]}$ be a family non-empty $r$-uniform hypergraphs. We define $\bigsqcup_{i \in [m]} H_i$ to be the hypergraph $G$ with vertex set $\bigcup_{i \in [m]} V(H_i) \times \{i\}$ and edge set
    \[
        E(G) = \bigcup_{i \in [m]} \{\, e \times \{i\} \mid e \in E(H_i) \,\}.
    \]
\end{definition}

\begin{proposition}\label{prop:wsat-family-equal-wsat-disjoint}
    Let $r \ge 1$ be an integer and let $\mc{H}$ be a non-empty family of non-empty $r$-uniform hypergraphs. Then
    \[
        \wsat(n, \mc{H}) = \wsat\bigl(n, \bigsqcup_{H \in \mc{H}} H\bigr) + \Theta(1).
    \]
\end{proposition}

\begin{proof}
    Let $G = \bigsqcup_{H \in \mc{H}} H$.

    The bound $\wsat(n, \mc{H}) \le \wsat(n, G)$ is immediate from the fact that $\Wsat(n, G) \subseteq \Wsat(n, \mc{H})$.

    We now prove the bound in the opposite direction. Fix $n \ge |V(G)|$ and take a hypergraph $F \in \Wsat(n, \mc{H})$ such that $|E(F)| = \wsat(n, \mc{H})$. Fix a subset $Z \subseteq V(F)$ of size $|V(G)|$. Let $F'$ be the hypergraph on the vertex set $V(F) = V(F')$ with edge set
    \[
        E(F') = E(F) \cup \binom{Z}{r}.
    \]
    Then $F'$ is weakly $G$-saturated, since the missing edges can be added in the same order as in $F$, and if at a given step a copy of some $H \in \mc{H}$ is created in $F$, then in $F'$ one can form a copy of $G$ by taking an arbitrary subset of $Z \setminus V(H)$ as the preimage of $V(G) \setminus V(H)$.

    Therefore,
    \[
        \wsat(n, G) \le \wsat(n, \mc{H}) + \binom{|V(G)|}{r}.
    \]
\end{proof}

\subsection{Degenerate cases}

In the case $r = 1$, the following explicit expression for $\wsat(n,\mc{H})$ is straightforward to obtain.

\begin{proposition}\label{prop:wsat-behavior-r-1}
    Let $\mc{H}$ be a non-empty family of $1$-uniform hypergraphs. Then for all $n \ge \max\{\, |V(H)| \mid H \in \mc{H} \,\}$,
    \[
        \wsat(n, \mc{H}) = \min\{\, |E(H)| \mid H \in \mc{H} \,\} - 1.
    \]
\end{proposition}

In the case $s(H) = 1$, the behavior of $\wsat(n, H)$ can be described as follows.

\begin{proposition}\label{prop:wsat-behavior-s-1}
    Let $r \ge 1$ be an integer and let $\mc{H}$ be a non-empty family of $r$-uniform hypergraphs such that $\min\{\, s(H) \mid H \in \mc{H} \,\} = 1$. Then there exists an integer $C_{\mc{H}}$ such that
    \[
        C_{\mc{H}} = \lim_{n \to \infty} \wsat(n, \mc{H}).
    \]
    Moreover, for all $n \ge \max\{\, |V(H)| \mid H \in \mc{H} \,\}$,
    \[
        \wsat(n, \mc{H}) \ge C_{\mc{H}}.
    \]
\end{proposition}

\begin{proof}
    By Proposition~\ref{thrm:tuza-theta-wsat}, it suffices to show that for all $n \ge \max\{\, |V(H)| \mid H \in \mc{H} \,\}$,
    \[
        \wsat(n+1, \mc{H}) \le \wsat(n, \mc{H}).
    \]

    Fix $n$ and a hypergraph $F \in \Wsat(n, \mc{H})$. Let $x \notin V(F)$ be a new vertex. Consider the hypergraph $\wt{F}$ on the vertex set $V(F) \cup \{x\}$ with edge set $E(\wt{F}) = E(F)$. We show that $\wt{F}$ is weakly $\mc{H}$-saturated. The edges is be added in the following order.
    First, add all the edges in $F$. Then, apply Proposition~\ref{prop:short-tuza-final-building-step} with $Z = V(F)$, choosing any hypergraph $H \in \mc{H}$ with $s(H) = 1$.
\end{proof}

\subsection{Trivial lower bounds}

The following quantities are generalizations of $\delta^*(H)$.
\begin{definition}
    Let $r \ge 1$ and $0 \le m \le r$ be integers, and let $H$ be an $r$-uniform hypergraph.

    For a non-empty $H$, define
    \[
        \delta_{m}(H) = \min\left\{\, |L_H(U)|\ \Big|\ U \in \binom{V(H)}{m},\ |L_H(U)| \neq 0 \,\right\}.
    \]

    For a empty $H$, set $\delta_{m}(H)=-1$.
\end{definition}

It is easy to see that $\delta_{r-1}(H) = \delta^*(H)$. There is also the following connection with sparseness.

\begin{proposition}\label{prop:sparseness-and-delta}
    Let $r \ge 1$ be an integer and let $H$ be a non-empty $r$-uniform hypergraph. Then
    \[
        s(H) = \min\left\{\, m \mid 0 \le m \le r,\ \delta_{m}(H) = 1 \,\right\}.
    \]
\end{proposition}

\begin{proof}
    For every $0 \le m \le r$, the condition $\delta_{m}(H) = 1$ means that there exists a subset $S \subseteq V(H)$ of size $m$ such that $\lbr{L_H(S)} = 1$, i.e., there exists exactly one edge $U \in E(H)$ containing $S$. The statement follows directly from the definition of $s(H)$.
\end{proof}

The following general lower bound holds.

\begin{proposition}\label{prop:lower-bound-via-delta-m}
    Let $r \ge 1$, $\delta\ge 1$ and $0 \le m \le r$ be integers. Let $H$ be an $r$-uniform hypergraph with $\delta_m(H)=\delta$. Then for all $n \ge |V(H)|$,
    \[
        \wsat(n, H) \ge \frac{\delta - 1}{\binom{r}{m}} \binom{n}{m}.
    \]
\end{proposition}

\begin{proof}
    Fix a hypergraph $F \in \Wsat(n, H)$ and a subset $U \in \binom{V(F)}{m}$. We want to show that $U$ is contained in at least $\delta - 1$ edges of $F$, i.e., $|L_F(U)| \ge \delta - 1$.

    If all edges containing $U$ are already present in $F$, then
    \[
        |L_F(U)| \ge \binom{n - m}{r - m} \ge \binom{|V(H)| - m}{r - m} \ge \delta.
    \]
    Otherwise, there exists some missing edge in $F$ that contains $U$. Take the first such edge $e$ that is added and contains $U$. Suppose it creates a copy $\wt{H}$ of the hypergraph $H$, then $L_{\wt{H}}(U) \setminus \{e \setminus U\}$ is a subset of $L_F(U)$, since $e$ is the first added edge containing $U$. Hence,
    \[
        |L_F(U)| \ge |L_{\wt{H}}(U)| - 1 = \delta - 1.
    \]

    The required lower bound on $|E(F)|$ follows from double counting:
    \[
        \binom{r}{m} \cdot |E(F)| = \sum_{U \in \binom{V(F)}{m}} |L_F(U)| \ge (\delta - 1) \cdot \binom{n}{m}.
    \]
\end{proof}

\section{Lower bound via polymatroids}\label{sec:linear-algebra-method}

\subsection{Matroids}

To formulate Kalai's linear algebraic method in full generality, we use matroids, which abstract the notion of linear independence. Matroids have many equivalent definitions; we use the definition via the rank function. Basic definitions and facts about matroids can be found in~\cite{2011Oxley}.

\begin{definition}
    Let $E$ be a finite set. A \emph{matroid} $M$ on $E$ is a pair $(E, \rk_M)$, where \mbox{$\rk_M : \pwrs{E} \to \mathbb{Z}$} is a function satisfying the following properties:
    \begin{gather*}
        \forall A \subseteq E\quad 0 \le \rk_M(A) \le |A|;\\
        \forall A \subseteq B \subseteq E\quad \rk_M(A) \le \rk_M(B);\\
        \forall A, B \subseteq E\quad \rk_M(A \cup B) + \rk_M(A \cap B) \le \rk_M(A) + \rk_M(B).
    \end{gather*}

    The rank of the matroid is defined as
    \[
        \rk(M) = \rk_M(E).
    \]
\end{definition}

\begin{remark*}
    In what follows, we consider matroids on the set of edges of the complete $r$-uniform hypergraph, i.e., $E = E(K_n^r)$. Therefore, for convenience, for a hypergraph $G \subseteq K_n^r$, we use $\rk_M(G)$ as a synonym for $\rk_M(E(G))$.
\end{remark*}

\begin{definition}
    Let $M = (E, \rk_M)$ be a matroid.

    A set $A \subseteq E$ is said to be \emph{independent} in $M$ if
    \[
        \rk_M(A) = |A|.
    \]

    A set $C \subseteq E$ is called a \emph{circuit} of $M$ if $\rk_M(C) < |C|$ and for every $e \in C$, the set $C \setminus \{e\}$ is independent.
\end{definition}

The following standard fact relates the circuits of a matroid to its rank.
\begin{proposition}\label{prop:matroid-rank-via-circuits}
    Let $M = (E, \rk_M)$ be a matroid. Then for every non-empty set $A \subseteq E$,
    \[
        \rk_M(A) = \max_{B \subseteq A} \left\{\, |B| \ \Big|\, \nexists\text{ circuit } C \subseteq B \,\right\}.
    \]
\end{proposition}

\subsection{Lower bound}

The idea of using matroids to obtain lower bounds on weak saturation numbers was formulated in full generality by Kalai~\cite{1985Kalai}. This bound use matroids that satisfy some compatibility conditions with the pattern hypergraph $H$.
\begin{definition}
    Let $r\ge 1$ and $n\ge 1$ be integers, and let $H$ be a non-empty $r$-uniform hypergraph.

    A matroid $M$ on the set $E(K_n^r)$ with rank function $\rk_M$ is called \emph{weakly $H$-saturated} if for every copy $\wt{H} \subseteq K_n^r$ of the hypergraph $H$, the following holds:
    \[
        \forall e \in E(\wt{H})\quad \rk_M(\wt{H} \setminus \{e\}) = \rk_M(\wt{H}).
    \]
\end{definition}

\begin{theorem}[\cite{1985Kalai}]\label{thrm:matroid-lower-bound}
    Let $r \ge 1$ and $n \ge 1$ be integers, and let $H$ be a non-empty $r$-uniform hypergraph. Let $M$ be a weakly $H$-saturated matroid on $E(K_n^r)$. Then
    \[
        \wsat(n, H) \ge \rk(M).
    \]
\end{theorem}

As mentioned in the introduction, it was shown in~\cite{2025TerekhovLinearAlgebra} that, in the graph case, the best bound on $\wsat(n,H)$ obtainable by this method has an integer asymptotic coefficient. This restriction can be avoided by working with $1$-polymatroids~\cite{2023Bonin} rather than matroids.

\begin{definition}
    Let $E$ be a finite set. An \emph{$1$-polymatroid} $M$ on $E$ is a pair $(E, \rho_M)$, where $\rho_M : \pwrs{E} \to \mathbb{R}$ is a function satisfying the following properties:
    \begin{gather}
        \forall A \subseteq E\quad 0 \le \rho_M(A) \le |A|;\label{eq:poly-boundness}\\
        \forall A \subseteq B \subseteq E\quad \rho_M(A) \le \rho_M(B);\label{eq:poly-monotone}\\
        \forall A, B \subseteq E\quad \rho_M(A \cup B) + \rho_M(A \cap B) \le \rho_M(A) + \rho_M(B).\label{eq:poly-submodular}
    \end{gather}
\end{definition}

\begin{definition}
    Let $r\ge 1$ and $n\ge 1$ be integers, and let $H$ be a non-empty $r$-uniform hypergraph. A $1$-polymatroid $M$ on the set $E(K_n^r)$ with rank function $\rho_M$ is called \emph{weakly $H$-saturated} if for every copy $\wt{H} \subseteq K_n^r$ of the hypergraph $H$, the following holds:
    \begin{equation}\label{eq:condition-poly}
        \forall e \in E(\wt{H})\quad \rho_M(\wt{H} \setminus \{e\}) = \rho_M(\wt{H}).
    \end{equation}
\end{definition}

\begin{theorem}\label{thrm:poly-lower-bound}
    Let $r \ge 1$ and $n \ge 1$ be integers, and let $H$ be a non-empty $r$-uniform hypergraph. Let $M$ be a be a weakly $H$-saturated $1$-polymatroid on the set $E(K_n^r)$ with rank function $\rho_M$. Then
    \[
        \wsat(n, H) \ge \rho_M(K_n^r).
    \]
\end{theorem}

\begin{proof}
    The proof is analogous to the proof of Theorem~\ref{thrm:matroid-lower-bound}.

    Fix a hypergraph $F \in \Wsat(n, H)$ such that $|E(F)| = \wsat(n, H)$. By definition, the edges of the complement $e_1, \ldots, e_k \in E(K_n^r) \setminus E(F)$ can be ordered such that for every $i \in [k]$ there exists a copy $\wt{H}_i \subseteq F \cup \{e_1, \ldots, e_i\}$ of $H$ with $e_i \in E(\wt{H}_i)$. From condition \eqref{eq:condition-poly} and the submodularity condition \eqref{eq:poly-submodular} with $A = F \cup \{e_1, \ldots, e_{i-1}\}$ and $B = \wt{H}$, we get
    \[
        \forall i \in [k] \quad \rho_M(F \cup \{e_1, \ldots, e_i\}) = \rho_M(F \cup \{e_1, \ldots, e_{i-1}\}),
    \]
    and thus,
    \[
        \rho_M(K_n^r) = \rho_M(F) \le |E(F)| = \wsat(n, H).
    \]
\end{proof}

\begin{remark*}
    In the definition of a \emph{polymatroid}~\cite{1978Fujishige,2003Edmonds}, compared with the definition of a \mbox{$1$-polymatroid}, condition \eqref{eq:poly-boundness} is relaxed by dropping the upper bound $\rho_M(A)\le |A|$ and requiring only $\rho_M(\emptyset)=0$.

    The extra upper bound $\rho_M(A)\le |A|$ in definition of \mbox{$1$-polymatroids} is a merely a convenient normalization, since scaling $\rho_M$ by a positive constant preserves all polymatroid axioms, and \eqref{eq:condition-poly} is invariant under such scaling.
    Hence, restricting to \mbox{$1$-polymatroids} causes no loss of generality, and we refer to Theorem~\ref{thrm:poly-lower-bound}
    as the polymatroidal method.
\end{remark*}

\section{Count polymatroids}\label{sec:count-polymatroids}

\subsection{Kruskal-Katona theorem}

To define count matroids, we need the concept of the shadow of a hypergraph.

\begin{definition}
    Let $r\ge 1$ and $m\ge 1$ be integers, and let $H$ be an $r$-uniform hypergraph. Define the \emph{$m$-shadow} of $H$ as
    \[
        \shadow{m}(H) = \left\{ U \in \binom{V(H)}{m}\ \middle|\ \exists e \in E(H),\ U \subseteq e \right\}.
    \]
\end{definition}

One of the main technical tools for bounding the size of the shadow of a hypergraph is the Kruskal--Katona theorem~\cites{1963Kruskal}{1968Katona}{1984Frankl}, which provides a lower bound in terms of \emph{left-compressed} hypergraphs.

\begin{definition}
    Let $e \ge 1$ and $r \ge 1$ be integers.

    We define inductively an $r$-uniform hypergraph with $e$ edges, denoted by $\lcG(r, e)$, called the \emph{left-compressed} hypergraph.

    For $r = 1$, set $V(\lcG(1, e))=[e]$, $E(\lcG(1, e)) = \binom{[e]}{1}$.

    For $r \ge 2$, let $k$ be the largest integer such that $e \ge \binom{k}{r}$. If $m = e - \binom{k}{r}$ is equal to zero, then set $V(\lcG(r, e))=[k]$, $E(\lcG(r, e)) = \binom{[k]}{r}$; otherwise, set $V(\lcG(r, e))=[k+1]$ and
    \[
        E(\lcG(r, e)) = \binom{[k]}{r} \cup \left\{ U \cup \{k+1\} \,\Big|\, U \in E(\lcG(r-1, m)) \right\}.
    \]
\end{definition}

\begin{theorem}[\cite{1963Kruskal, 1968Katona}]\label{thrm:kruskal-katona}
    Let $r \ge 1$, $e\ge 1$, and $0\le m \le r$ be integers. Let $F$ be an $r$-uniform hypergraph with $e$ edges. Then
    \[
        |\shadow{m}(F)| \ge |\shadow{m}(\lcG(r, e))|.
    \]
\end{theorem}

\subsection{Count Matroids}

To define count matroids, we first need the notion of a submodular function.
\begin{definition}
    Let $E$ be a set and let $L \colon \pwrs{E} \to \mathbb{R}$ be a function.

    The function $L$ is called \emph{non-decreasing} if
    \[
        \forall A \subseteq B \subseteq E \quad L(A) \le L(B).
    \]

    The function $L$ is called \emph{submodular} if
    \[
        \forall A, B \subseteq E \quad L(A \cup B) + L(A \cap B) \le L(A) + L(B).
    \]
\end{definition}

The following result from~\cite[Proposition 11.1.1]{2011Oxley} (together with Proposition~\ref{prop:matroid-rank-via-circuits}) allows one to obtain matroids from submodular functions.

\begin{proposition}\label{prop:submodular-to-matroid}
    Let $E$ be a finite set and let $L : \pwrs{E} \to \mathbb{Z}$ be a submodular, non-decreasing function. Then the function $\rk_N$, defined for non-empty $A \subseteq E$ as
    \[
        \rk_N(A) = \max_{B \subseteq A} \left\{\, |B| \ \Big|\ \forall C \subseteq B,\ C \neq \emptyset:\ |C| \le L(C) \,\right\},
    \]
    defines a matroid $N = (E, \rk_N)$.
\end{proposition}

We will apply Proposition~\ref{prop:submodular-to-matroid} to a function that is defined in terms of multihypergraphs.

\begin{definition}
    Let $q \ge 1$ and $r \ge 1$ be integers. For an $r$-uniform hypergraph $H$, let $\mulg{H}{q}$ denote the multihypergraph on the vertex set $V(H)$ with edges
    \[
        E(\mulg{H}{q}) = \left\{\, (e, i) \mid e \in E(H),\ i \in [q] \,\right\}.
    \]
\end{definition}

\begin{definition}
    Let $q \ge 1$, $r \ge 1$ and $n \ge 1$ be integers. For a multihypergraph $G \subseteq \mulg{K_n^r}{q}$, define its \emph{underlying hypergraph} as
    \[
        \pi(G) = \left\{\, e \in E(K_n^r) \mid \exists i \in [q]\ (e, i) \in E(G) \,\right\}.
    \]
    For an ordinary hypergraph $G \subseteq K_n^r$, we define $\pi(G) = G$.
\end{definition}

The following function will be used in Proposition~\ref{prop:submodular-to-matroid} and is analogous to the one introduced in~\cite{2001PikhurkoCount}.
\begin{definition}
    Let $r \ge 1$, $q \ge 1$ and $n \ge 1$ be integers, and let $a_0, a_1, \ldots, a_r$ be real numbers. Define the function $L_{\mb{a}} \colon \pwrs{E(\mulg{K_n^r}{q})} \to \mathbb{R}$ by
    \[
        \forall G \subseteq \mulg{K_n^r}{q} \quad L_{\mb{a}}(G) = a_0 + \sum_{i=1}^{r} a_i \cdot |\shadow{i}(\pi(G))|.
    \]
\end{definition}

\begin{proposition}\label{prop:L-a-is-submodular}
    Let $r \ge 1$, $q \ge 1$ and $n \ge 1$ be integers, and let $a_0, a_1, \ldots, a_r$ be integers. Suppose that for all $i\ge 1$ we have $a_i \ge 0$. Then the function $L_{\mb{a}} \colon \pwrs{E(\mulg{K_n^r}{q})} \to \mathbb{R}$ is integer-valued, non-decreasing, and submodular.
\end{proposition}

\begin{proof}
    It is straightforward to show that $|\shadow{i}(\pi(G))|$ is a non-decreasing and submodular function for all $i \ge 1$. Therefore, every linear combination of such functions with non-negative coefficients is also non-decreasing and submodular. Adding the constant $a_0$ does not affect these properties.

    Integrality immediately follows from the integrality of the coefficients $a_i$.
\end{proof}

In view of Propositions~\ref{prop:L-a-is-submodular} and~\ref{prop:submodular-to-matroid}, the functions $L_{\mb{a}}$ give rise to a family of matroids, which we call \emph{count matroids}.

\begin{definition}\label{def:count-matroid}
    Let $r \ge 1$, $q \ge 1$ and $n \ge 1$ be integers. Let $a_0, a_1, \ldots, a_r$ be integers such that for all $i\ge 1$ we have $a_i \ge 0$. Denote by $M(n, r, q, \mb{a})$ the matroid on $E(\mulg{K_n^r}{q})$ obtained by applying Proposition~\ref{prop:submodular-to-matroid} to the function $L_{\mb{a}} \colon \pwrs{E(\mulg{K_n^r}{q})} \to \mathbb{R}$.
\end{definition}

We need the following properties of count matroids.

\begin{proposition}\label{prop:count-matroid-bound-via-L}
    Let $r,q,a_0, a_1, \ldots, a_r$ and $n$ be as in Definition~\ref{def:count-matroid}. Then for every non-empty $A \subseteq E(\mulg{K_n^r}{q})$,
    \[
        \rk_{M(n, r, q, \mb{a})}(A) \le \max(0, L_{\mb{a}}(A)).
    \]
\end{proposition}

\begin{proof}
    If $\rk(A) \neq 0$, then by definition there exists a subset $B \subseteq A$ such that $1 \le |B| \le L(B)$ and
    \(
        \rk(A) = |B|.
    \)
    The claim now follows from the monotonicity of $L$.
\end{proof}

\begin{proposition}\label{prop:count-matroid-dep-less-L}
    Let $r,q,a_0, a_1, \ldots, a_r$ and $n$ be as in Definition~\ref{def:count-matroid}. Let $A \subseteq E(\mulg{K_n^r}{q})$ be a non-empty set such that $\rk_{M(n, r, q, \mb{a})}(A) < |A|$. Then there exists a non-empty subset $B \subseteq A$ such that
    \[
        |B| > L(B).
    \]
\end{proposition}

\begin{proof}
    Suppose that for every non-empty $B \subseteq A$ we have $|B| \le L(B)$. Then by the definition of $\rk_M$ from Proposition~\ref{prop:submodular-to-matroid}, $\rk_M(A) = |A|$, which contradicts the assumption.
\end{proof}

\begin{proposition}\label{prop:cirucit-and-L-count-matroid}
    Let $r,q,a_0, a_1, \ldots, a_r$ and $n$ be as in Definition~\ref{def:count-matroid}. Let $C \subseteq E(\mulg{K_n^r}{q})$ be a circuit of the matroid $M(n, r, q, \mb{a})$ with $|C| \ge 2$. Then for every element $e \in C$, we have
    \[
        L_{\mb{a}}(C) = L_{\mb{a}}(C \setminus \{e\}) = |C \setminus \{e\}|.
    \]
\end{proposition}

\begin{proof}
    By Proposition~\ref{prop:count-matroid-dep-less-L}, there exists a non-empty $B \subseteq C$ such that $|B| > L(B)$. If $B \neq C$, then $B$ is independent and by Proposition~\ref{prop:count-matroid-bound-via-L}, $|B| = \rk_M(B) \le L(B)$, a contradiction. Hence $C > L(C)$.

    From Proposition~\ref{prop:count-matroid-bound-via-L},
    \begin{equation*}
        1 \le |C \setminus \{e\}| = \rk_M(C \setminus \{e\}) \le L(C \setminus \{e\}) \le L(C) \le |C| - 1 = |C \setminus \{e\}|.
    \end{equation*}
\end{proof}

To determine the rank of a count matroid, we need the concept of connectivity.

\begin{definition}
    Let $N = (E, \rk_N)$ be a matroid. Define the relation $\xi_N$ on $E$ as follows: for $a, b \in E$,
    \[
        \matcon{N}{a}{b} \quad \iff \quad \exists \text{ circuit } C: a \in C \text{ and } b \in C.
    \]

    We say that $N$ is \emph{connected} if for every two distinct elements $a, b \in E$, we have $\matcon{N}{a}{b}$.
\end{definition}

To check connectivity, it is convenient to use the transitivity of $\xi_N$~\cite[Proposition 4.1.2]{2011Oxley}.

\begin{proposition}[\cite{2011Oxley}]\label{prop:matroid-connectivity-is-trasitive}
    Let $N = (E, \rk_N)$ be a matroid and let $a, b \in E$ be two distinct elements. Suppose that there exists a sequence $a = a_0, a_1, \ldots, a_s = b$ such that $\forall i \ge 1,\ \matcon{N}{a_{i-1}}{a_i}$. Then $\matcon{N}{a}{b}$.
\end{proposition}

Count matroids are connected in a wide range of cases.

\begin{proposition}\label{prop:count-matroid-is-connected}
    Let $r,q,a_0, a_1, \ldots, a_r$ and $n$ be as in Definition~\ref{def:count-matroid}. Let $p = a_0 + \sum_{i=1}^r a_i \cdot \binom{r}{i}$. Suppose that $p \ge q$ and $\rk(M(n, r, q, \mb{a})) < q \cdot \binom{n}{r}$. Then $M(n, r, q, \mb{a})$ is a connected matroid.
\end{proposition}

\begin{proof}
    Let $E = E(\mulg{K_n^r}{q})$ and $M = M(n, r, q, \mb{a})$.

    Since $E$ is not independent, there exists a circuit $C$ in the matroid $M$, and since $p \ge 1$, we have $|C| \ge 2$.

    For the circuit $C$, perform the following procedure.

    Choose $G \subseteq E$ such that $|G| = |C|$, $\pi(G)\seq\pi(C)$, $\pi(G) \simeq  \lcG(r, |\pi(G)|)$, and all edges in $\pi(G)$, possibly except one, appear in $G$ with multiplicity $q$. By Theorem~\ref{thrm:kruskal-katona} and Proposition~\ref{prop:cirucit-and-L-count-matroid},
    \[
        L(G) \le L(C) < |C| = |G|,
    \]
    so $G$ is not independent. Take a subset $\wt{C} \subseteq G$ which is a circuit.

    Repeat the procedure, setting $C := \wt{C}$ at each step, until the circuit stabilizes. In the end, we obtain a circuit $\wt{C}$ such that $\pi(\wt{C}) \simeq \lcG(r, |\pi(\wt{C})|)$ and all edges of $\pi(\wt{C})$, except possibly one, occur in $\wt{C}$ with multiplicity $q$.

    Since $p \ge q$, we have $|\pi(\wt{C})| \ge 2$. By the definition of a left-compressed hypergraph, there exist two edges $e_1, e_2 \in \wt{C}$ such that $|\pi(e_1) \cap \pi(e_2)| = r - 1$; for instance, take $\pi(e_1) = [r]$ and $\pi(e_2) = [r+1] \setminus \{r-1\}$. Without loss of generality, assume $\pi(e_1)$ appears in $\wt{C}$ with multiplicity $q$. By the definition of count matroids, every $A \subseteq E$ with $|A| = |\wt{C}|$ and $\pi(A) \simeq \pi(\wt{C})$ is also a circuit.

    Now, fix two distinct elements $e, e' \in E$. We want to show that $\matcon{M}{e}{e'}$.

    If $\pi(e) = \pi(e')$, choose a circuit $A \subseteq E$ such that $|A| = |C|$, $\pi(A) \simeq \pi(\wt{C})$, and $\pi(\wt{e}_1) = \pi(e)$, where $\pi(\wt{e}_1)$ is the preimage in $\pi(A)$ of the edge $\pi(e_1)$ from $\pi(C)$.

    If $\pi(e) \ne \pi(e')$ and $|\pi(e) \cap \pi(e')| = r - 1$, choose a circuit $A \subseteq E$ such that $e, e' \in A$, $|A| = |C|$, $\pi(A) \simeq \pi(\wt{C})$, $\pi(\wt{e}_1) = \pi(e)$, and $\pi(\wt{e}_2) = \pi(e')$, where $\pi(\wt{e}_1)$ and $\pi(\wt{e}_2)$ are the preimages in $\pi(A)$ of $\pi(e_1)$ and $\pi(e_2)$, respectively.

    In the remaining case, where $\pi(e) \ne \pi(e')$ and $|\pi(e) \cap \pi(e')| \le r - 2$, note that any two distinct edges of $E(K_n^r)$ can be connected by a sequence of edges in which each consecutive pair shares $r - 1$ vertices. Hence, there exists a sequence $e = e_0, e_1, \ldots, e_s = e'$ such that
    \[
        \forall i \ge 1 \quad |\pi(e_{i-1}) \cap \pi(e_i)| = r - 1.
    \]
    By the previous case, for all $i \ge 1$ we have $\matcon{M}{e_{i-1}}{e_i}$, and by Proposition~\ref{prop:matroid-connectivity-is-trasitive}, we conclude that $\matcon{M}{e}{e'}$.
\end{proof}

The following theorem determines the rank of a count matroid. In the case of ordinary hypergraphs (rather than multihypergraphs), this was established by Pikhurko~\cite{2001PikhurkoCount}. Our proof strategy is similar, although we use matroid connectivity to encapsulate some of the technical details.

\begin{theorem}\label{thrm:rank-count-matroid}
    Let $r,q,a_0, a_1, \ldots, a_r$ and $n$ be as in Definition~\ref{def:count-matroid}. Let $p = a_0\! +\! \sum_{i=1}^r a_i \cdot \binom{r}{i}$. Then
    \[
        \rk(M(n, r, q, \mb{a})) = \min\left\{ L_{\mb{a}}(\mulg{K_n^r}{q}),\ \min(q, \max(0, p)) \cdot \binom{n}{r} \right\}.
    \]
\end{theorem}

\begin{proof}

    Let $E = E(\mulg{K_n^r}{q})$, $M = M(n, r, q, \mb{a})$, and $L = L_{\mb{a}}$. By Proposition~\ref{prop:count-matroid-bound-via-L}, for every edge $e \in E$,
    \[
        \rk_M(e) \le  \max(0, L(\{e\})) = \max(0, p).
    \]
    If $p \le 0$, then $\rk(M) = 0$, so we may assume $p \ge 1$.

    If $p \le q$, then $L_{\mb{a}}(E) = L_{\mb{a}}(\mulg{K_n^r}{p})$ and for every edge $e \in E$ we have $\rk_M(\mulg{\{e\}}{q}) = p$. It follows that
    \[
        \rk(M) = \rk_M(\mulg{K_n^r}{p}) = \rk(M(n, r, p, \mb{a})).
    \]
    Thus, it suffices to prove the theorem for $p \ge q$.

    Suppose that
    \[
        \rk_M(E) < q \cdot \binom{n}{r}.
    \]
    Since $E$ is not independent, $M$ contains a circuit $C$. Because $p \ge q$, we have $|\pi(\wt{C})| \ge 2$.

    Take a maximal non-empty independent set $A \subseteq E$ such that $L(A) = |A|$. At least one such set exists due to Proposition~\ref{prop:cirucit-and-L-count-matroid} applied to a circuit $C$ and any $e \in C$.

    We claim that $\rk_M(A) = \rk(M)$. Suppose not. Then there exists $x \in E$ such that
    \begin{equation}\label{eq:choose-x-for-A}
        \rk_M(A \cup \{x\}) > \rk_M(A).
    \end{equation}

    Take any $a \in A$. By Proposition~\ref{prop:count-matroid-is-connected}, there exists a circuit $C$ containing both $a$ and $x$. Let $Y = C \setminus \{x\}$. Take an independent set $\wt{A}$ such that $A \subseteq \wt{A} \subseteq A \cup Y$ and $\rk_M(\wt{A}) = \rk_M(A \cup Y)$. Since $\rk_M(\wt{A} \cup \{x\}) = \rk_M(\wt{A})$, there exists a circuit $\wt{C} \subseteq \wt{A} \cup \{x\}$ such that $x \in \wt{C}$. Let $B = \wt{C} \setminus \{x\}$. Then $B$ is independent and $B \not\subseteq A$ by \eqref{eq:choose-x-for-A}. Also, $|B \cap A| \ge 1$, since otherwise $\wt{C} \subseteq Y \cup \{x\} \subseteq C \setminus \{a\}$, contradicting the fact that $C$ is a circuit.

    By Proposition~\ref{prop:count-matroid-bound-via-L}, submodularity of $L$, and Proposition~\ref{prop:cirucit-and-L-count-matroid}, we have
    \begin{align*}
        |A \cup B| \le L(A \cup B)
        &\le L(A \cup \wt{C}) \le L(A) + L(\wt{C}) - L(A \cap B) \\
        &= |A| + |B| - L(A \cap B) \le |A| + |B| - |A \cap B| = |A \cup B|,
    \end{align*}
    where the last inequality uses $|A \cap B| \le L(A \cap B)$, which holds by Proposition~\ref{prop:count-matroid-bound-via-L} since $\rk_M(A \cap B) = |A \cap B| \ge 1$. But then $A \cup B$ is independent and strictly larger than $A$, contradicting the maximality of $A$. Therefore, $\rk_M(A) = \rk(M)$.

    Since $\rk_M(A) = \rk(M)$, for every $e \in E \setminus A$, there exists a circuit $C_e$ such that $C_e \setminus \{e\} \subseteq A$. Using Proposition~\ref{prop:cirucit-and-L-count-matroid}, Proposition~\ref{prop:count-matroid-bound-via-L}, and submodularity of $L$,
    \[
        L(A \cup \{e\}) \le L(A) + L(C_e) - L(A \cap C_e) \le L(A) + |C_e \setminus \{e\}| - |A \cap C_e| = L(A).
    \]
    By submodularity of $L$,
    \[
        L(E) + (|E \setminus A| - 1) \cdot L(A) \le \sum_{e \in E \setminus A} L(A \cup \{e\}) = |E \setminus A| \cdot L(A).
    \]
    Hence,
    \[
        \rk(M) = \rk_M(A) = L(A) = L(E).
    \]
\end{proof}

\subsection{Count Polymatroids}

Count polymatroids generalize count matroids by allowing rational coefficients in $L_{\mb{a}}$.

\begin{proposition}\label{prop:count-polymatroids}
    Let $r \ge 1$ and $n \ge 1$ be integers. Let $a_0, a_1, \ldots, a_r$ be rational numbers such that for all $i\ge 1$ we have $a_i \ge 0$. Let $q \ge 1$ be the smallest positive integer such that \mbox{$\forall i \ge 0:\ q \cdot a_i \in \mathbb{Z}$}. Let $q\cdot \mb{a}$ denote $(q \cdot a_0, \ldots, q \cdot a_r)$ and let $N = M(n, r, q, q\cdot \mb{a})$. Define a function $\rho_M : \pwrs{E(K_n^r)} \to \mathbb{Q}$ by
    \[
        \forall G \subseteq K_n^r \quad \rho_M(G) = \frac{\rk_N(\mulg{G}{q})}{q}.
    \]
    Then $M = (E(K_n^r), \rho_M)$ is a $1$-polymatroid.
\end{proposition}

\begin{proof}
    Properties \eqref{eq:poly-boundness} and \eqref{eq:poly-monotone} follow easily from the corresponding properties of matroids and the fact that $|E(\mulg{G}{q})| = q \cdot |E(G)|$.

    To establish property~\eqref{eq:poly-submodular}, note that for any hypergraphs $G_1, G_2 \subseteq K_n^r$,
    \[
        E(\mulg{G_1}{q}) \cup E(\mulg{G_2}{q}) = E(\mulg{G_1 \cup G_2}{q}).
    \]
\end{proof}

\begin{definition}\label{def:count-polymatroid}
    Let $r \ge 1$ and $n \ge 1$ be integers. Let $a_0, a_1, \ldots, a_r$ be rational numbers such that for all $i\ge 1$ we have $a_i \ge 0$. Denote by $M(n, r, \mb{a})$ the $1$-polymatroid obtained from Proposition~\ref{prop:count-polymatroids} for these parameters. This $1$-polymatroid is called a \emph{count polymatroid}.
\end{definition}

The following theorem determines the rank of a count polymatroid.
\begin{theorem}\label{thrm:rank-count-poly}
    Let $r,a_0, a_1, \ldots, a_r$ and $n$ be as in Definition~\ref{def:count-polymatroid}. Then
    \[
        \rho_{M(n, r, \mb{a})}(K_n^r) = \min\left\{ L_{\mb{a}}(K_n^r),\ \min(1, \max(0, p)) \cdot \binom{n}{r} \right\},
    \]
    where $p = a_0 + \sum_{i=1}^r a_i \cdot \binom{r}{i}$.
\end{theorem}

\begin{proof}
    This follows immediately from Theorem~\ref{thrm:rank-count-matroid} and the fact that
    \[
        L_{\mb{a}}(\mulg{K_n^r}{q}) = L_{\mb{a}}(\pi(\mulg{K_n^r}{q})) = L_{\mb{a}}(K_n^r).
    \]
\end{proof}

We obtain the following corollary for weak saturation numbers.

\begin{corollary}\label{corollary:wsat-lower-bound-via-count-poly}
    Let $r,a_0, a_1, \ldots, a_r$ and $n$ be as in Definition~\ref{def:count-polymatroid}, and let $H$ be a non-empty $r$-uniform hypergraph. Suppose that $a_0 + \sum_{i=1}^r a_i \cdot \binom{r}{i} \ge 1$ and for every copy $\wt{H} \subseteq K_n^r$ of the hypergraph $H$ the following holds:
    \begin{equation}\label{eq:wsat-condition-count-poly}
        \forall e \in E(\wt{H})\quad \rho_{M(n, r, \mb{a})}(\wt{H} \setminus \{e\}) = \rho_{M(n, r, \mb{a})}(\wt{H}).
    \end{equation}
    Then for all $n \ge |V(H)|$,
    \[
        \wsat(n, H) \ge L_{\mb{a}}(K_n^r).
    \]
\end{corollary}

\begin{proof}
    From Theorem~\ref{thrm:poly-lower-bound} and Theorem~\ref{thrm:rank-count-poly}, we have
    \[
        \wsat(n, H) \ge \min\left\{ L_{\mb{a}}(K_n^r),\ \binom{n}{r} \right\}.
    \]
    But for $n \ge |V(H)|$, it holds that $\wsat(n, H) < \binom{n}{r}$, so the result follows.
\end{proof}

The following propositions are useful for verifying condition \eqref{eq:wsat-condition-count-poly}.

\begin{proposition}\label{prop:count-poly-less-L}
    Let $r,a_0, a_1, \ldots, a_r$ and $n$ be as in Definition~\ref{def:count-polymatroid}. Then for every $G \subseteq K_n^r$,
    \[
        \rho_{M(n, r, \mb{a})}(G) \le \max(0, L_{\mb{a}}(G)).
    \]
\end{proposition}

\begin{proof}
    This follows from Proposition~\ref{prop:count-matroid-bound-via-L}.
\end{proof}

\begin{proposition}\label{prop:count-poly-indep-criterion}
    Let $r,a_0, a_1, \ldots, a_r$ and $n$ be as in Definition~\ref{def:count-polymatroid}. Let $F \subseteq K_n^r$ be a non-empty $r$-uniform hypergraph such that $\rho_{M(n, r, \mb{a})}(F) < |E(F)|$. Then there exists a non-empty $r$-uniform hypergraph $G \subseteq F$ such that
    \[
        L_{\mb{a}}(G) < |E(G)|.
    \]
\end{proposition}

\begin{proof}
    Let $q \ge 1$ be the smallest positive integer such that $\forall i \ge 0:\ q \cdot a_i \in \mathbb{Z}$. Let $N = M(n, r, q, q\cdot \mb{a})$.

    Since $\rk_N(\mulg{F}{q}) < q \cdot |F| = |\mulg{F}{q}|$, then by Proposition~\ref{prop:count-matroid-dep-less-L}, there exists non-empty $\wt{G} \subseteq \mulg{F}{q}$ such that $L_{q\cdot \mb{a}}(\wt{G}) < |E(\wt{G})|$. Let $G = \pi(\wt{G})$. Then
    \[
        q\cdot L_{\mb{a}}(G)=L_{q\cdot \mb{a}}(\pi(\wt{G})) = L_{q\cdot \mb{a}}(\wt{G}) < |E(\wt{G})| \le q \cdot |E(\pi(\wt{G}))| = q\cdot |E(G)|,
    \]
    hence $G$ satisfies the required inequality.
\end{proof}

\section{General Lower bound}\label{sec:gamma-lower-bound-hyper}

In view of Theorem~\ref{thrm:tuza-theta-wsat}, and for the sake of generality, our general bound is of the form $\wsat(n,H) = \Omega(n^{s(H)-1})$. The bound is expressed in terms of the quantity $\gamma_{s,H}$. This quantity serve as the coefficient of $a_{s-1}$ for count polymatroids, and it is chosen so as to make it convenient to verify that the resulting polymatroid is weakly $H$-saturated. A more combinatorial natural definition of this quantity is given in Proposition~\ref{prop:alt-form-gamma}.

\begin{definition}
    Let $r \ge 2$ and $s\ge 2$ be integers. Let $H$ be an $r$-uniform hypergraph $H$ with $s(H) = s$. Define
    \[
        \gamma_{s,H} = \min\left\{\, \frac{|E(H)| - |E(G)| - 1}{|\shadow{s-1}(H)| - |\shadow{s-1}(G)|}\ \middle|\ G \subseteq H,\ E(G) \neq \emptyset,\ |\shadow{s-1}(G)| < |\shadow{s-1}(H)| \,\right\}.
    \]
\end{definition}

\begin{theorem}\label{thrm:gamma-lower-bound-hyper}
    Let $r \ge 2$ and $s\ge 2$ be integers. Let $H$ be an $r$-uniform hypergraph $H$ with $s(H) = s$. Then for all $n \ge |V(H)|$,
    \[
        \wsat(n, H) \ge \gamma_{s,H} \cdot \left( \binom{n}{s-1} - |\shadow{s-1}(H)| \right) + |E(H)| - 1.
    \]
\end{theorem}

\begin{proof}
    Choose $a_0, \ldots, a_r \in \mathbb{Q}$ such that
    \[
        \forall G \subseteq K_n^r \quad L_{\mb{a}}(G) = \gamma_{s,H} \cdot \left( |\shadow{s-1}(G)| - |\shadow{s-1}(H)| \right) + |E(H)| - 1,
    \]
    i.e., set $a_{s-1} = \gamma_{s,H}$, $a_0 = |E(H)| - 1 - \gamma_{s,H} \cdot |\shadow{s-1}(H)|$, and $a_i = 0$ for all other $i$.

    Taking $G = \{e\}$ for some edge $e \in E(H)$ in the definition of $\gamma_{s,H}$ yields:
    \[
        \frac{|E(H)| - 2}{|\shadow{s-1}(H)| - \binom{r}{s-1}} \ge \gamma_{s,H},
    \]
    from which it follows that $a_{s-1} \cdot \binom{r}{s-1} + a_0 \ge 1$.

    By Corollary~\ref{corollary:wsat-lower-bound-via-count-poly}, it suffices to show that for every copy $\wt{H} \subseteq K_n^r$ of $H$ and every edge $e \in E(\wt{H})$, we have
    \[
        \rho_{M(n, r, \mb{a})}(\wt{H} \setminus \{e\}) = \rho_{M(n, r, \mb{a})}(\wt{H}).
    \]

    Fix a copy $\wt{H}$ and an edge $e \in E(\wt{H})$.

    By Proposition~\ref{prop:count-poly-less-L},
    \[
        \rho_{M(n, r, \mb{a})}(\wt{H}) \le L_{\mb{a}}(\wt{H}) = |E(H)| - 1.
    \]

    So it is enough to prove that
    \[
        \rho_{M(n, r, \mb{a})}(\wt{H} \setminus \{e\}) = |E(H)| - 1.
    \]

    Suppose not. Since $s(H) \ge 2$, we have $|E(H)| \ge 2$. Then by Proposition~\ref{prop:count-poly-indep-criterion}, there exists a non-empty $G \subseteq \wt{H} \setminus \{e\}$ such that $|E(G)| > L_{\mb{a}}(G)$. If $|\shadow{s-1}(G)| = |\shadow{s-1}(H)|$, then
    \[
        L_{\mb{a}}(G) = |E(H)| - 1 \ge |E(G)|,
    \]
    which contradicts the choice of $G$. Therefore, $|\shadow{s-1}(G)| < |\shadow{s-1}(H)|$.

    From the definition of $\gamma_{s,H}$, it follows that
    \begin{equation*}
        \gamma_{s,H} \cdot (|\shadow{s-1}(H)| - |\shadow{s-1}(G)|) \le (|E(H)| - 1) - |E(G)|.
    \end{equation*}
    Hence,
    \[
        |E(G)| \le \gamma_{s,H} \cdot (|\shadow{s-1}(G)| - |\shadow{s-1}(H)|) + |E(H)| - 1 = L_{\mb{a}}(G),
    \]
    contradicting the choice of $G$.
\end{proof}

To bound $\gamma_{s,H}$ we need the following identity.

\begin{proposition}\label{prop:alt-form-gamma}
    Let $r \ge 2$ and $s\ge 2$ be integers. Let $H$ be an $r$-uniform hypergraph $H$ with $s(H) = s$. Then
    \begin{multline*}
        \gamma_{s,H} = \min\left\{\, \frac{\left| \left\{ e \in E(H) \mid \binom{e}{s-1} \cap \mc{S} \neq \emptyset \right\} \right| - 1}{|\mc{S}|}\ \ \middle|\right.\\
        \left. \emptyset \neq \mc{S} \subseteq \shadow{s-1}(H),\ |\shadow{s-1}(H) \setminus \mc{S}| \ge \binom{r}{s-1} \,\right\}.
    \end{multline*}
\end{proposition}

\begin{proof}
    Denote the right-hand side by $\wt{\gamma}_{s,H}$.

    First we show that $\wt{\gamma}_{s,H} \le \gamma_{s,H}$.
    Suppose the minimum in $\gamma_{s,H}$ is attained at $G \subseteq H$. Let $\mc{S} = \shadow{s-1}(H) \setminus \shadow{s-1}(G)$. Then
    \[
        |\shadow{s-1}(H) \setminus \mc{S}| = |\shadow{s-1}(G)| \ge \binom{r}{s-1}.
    \]
    All edges $e \in E(H)$ for which $\binom{e}{s-1} \cap \mc{S} \neq \emptyset$ must lie in $H \setminus G$; otherwise, $\mc{S}$ would intersect $\shadow{s-1}(G)$, contradicting the definition of $\mc{S}$. Therefore,
    \[
        \frac{|\{ e \in E(H) \mid \binom{e}{s-1} \cap \mc{S} \neq \emptyset \}| - 1}{|\mc{S}|} \le \frac{|E(H \setminus G)| - 1}{|\mc{S}|} = \gamma_{s,H}.
    \]

    Next we show that $\gamma_{s,H} \le \wt{\gamma}_{s,H}$.
    Suppose the minimum in $\wt{\gamma}_{s,H}$ is attained at $\mc{S} \subseteq \shadow{s-1}(H)$. Let $\wt{G} = \{ e \in E(H) \mid \binom{e}{s-1} \cap \mc{S} \neq \emptyset \}$.

    If $\wt{G} = H$, let $G$ consist of a single arbitrary edge of $H$. Since $|\mc{S}| \le |\shadow{s-1}(H)| - \binom{r}{s-1}$, we obtain:
    \[
        \wt{\gamma}_{s,H} = \frac{|E(H)| - 1}{|\mc{S}|} \ge \frac{|E(H)| - 2}{|\shadow{s-1}(H)| - \binom{r}{s-1}} = \frac{|E(H)| - |E(G)| - 1}{|\shadow{s-1}(H)| - |\shadow{s-1}(G)|} \ge \gamma_{s,H}.
    \]

    If $\wt{G} \neq H$, take $G = H \setminus \wt{G}$. By definition of $\wt{G}$, we have $\mc{S} \subseteq \shadow{s-1}(H) \setminus \shadow{s-1}(G)$, hence
    \[
        \wt{\gamma}_{s,H} = \frac{|E(\wt{G})| - 1}{|\mc{S}|} \ge \frac{|E(H)| - |E(G)| - 1}{|\shadow{s-1}(H)| - |\shadow{s-1}(G)|} \ge \gamma_{s,H}.
    \]
\end{proof}

\subsection{Connection with general bound for the graph case}\label{subsec:gamma-connection-with-graph-case}

For graphs ($r=2$), the optimal lower bound on $\wsat(n,H)$ in terms of $\delta(H)$ was derived in~\cite{2025TerekhovComb} from a general bound analogous to Theorem~\ref{thrm:gamma-lower-bound-hyper}, but it was proved using combinatorial arguments. This bound was formulated in terms of the following combinatorial quantity, which is closely related to $\gamma_{2,H}$.
\begin{definition}\label{def:gamma-graph}
    Let $H$ be a non-empty graph without isolated vertices. For an integer $0 \le m < |V(H)|$ define
    \begin{equation*}
        \gamma^m_{H}=\min\bigg\{\ \frac{\lbr{\fbr{e\in E(H)\mid e\cap U \neq \emptyset}}-1}{\lbr{U}}\ \bigg|\ \emptyset \neq U \seq V(H),\ \lbr{V(H)\setminus U}\ge m\bigg\}.
    \end{equation*}
\end{definition}
\begin{theorem}[\cite{2025TerekhovComb}]\label{thrm:gamma-lower-bound-graph}
    Let $H$ be a graph with $\delta(H)\ge 2$. Then for all $n\ge \lbr{V(H)}$,
    \[\wsat(n,H)\ge \gamma^1_H\cdot(n-\lbr{V(H)}) + \lbr{E(H)} - 1.\]
\end{theorem}

Proposition~\ref{prop:alt-form-gamma} implies that $\gamma_{2,H} = \gamma_H^2$. In view of this, we obtain the following strengthening of Theorem~\ref{thrm:gamma-lower-bound-graph}, which shows that the polymatroidal method yields an improvement over the bound obtained by combinatorial arguments.
\begin{corollary}\label{corollary:enhanced-gamma-upper-bound-graph}
    Let $H$ be a graph with $\delta(H) \ge 2$. Then for all $n \ge |V(H)|$,
    \[
        \wsat(n,H) \ge \gamma_H^2 \cdot (n - |V(H)|) + |E(H)| - 1.
    \]
\end{corollary}

A bound of this form was established in~\cite{2025TerekhovComb} for certain special cases, but not in general. In particular, it was not obtained for the clique $K_5$, for which $\gamma^1_{K_5} = \frac{9}{4} < \frac{8}{3} = \gamma^2_{K_5}$.

\section{Lower bound via $\delta^*(H)$}\label{sec:delta-lower-bound-hyper}

\subsection{Bounds on size of the shadow}

To derive Theorem~\ref{thrm:optimal-lower-bound-hyper} from Theorem~\ref{thrm:gamma-lower-bound-hyper}, we need lower bounds on the size of the shadow of an $r$-uniform hypergraph. To obtain these bounds, we will use the following inequalities for binomial coefficients.

\begin{proposition}\label{prop:binom-a-plus-b-ineq}
    Let $a \ge 0$ and $b \ge 0$ be integers. Then
    \[
        \binom{a+b}{a} \ge ab + 1.
    \]
\end{proposition}
\begin{proof}
    Without loss of generality, assume $b \ge a$. If $a = 0$, the inequality holds since both sides equal 1. From now on, assume $b \ge a \ge 1$.

    We prove the inequality by induction on $a + b$. For $a + b = 2$, the inequality holds because both sides equal $2$.

    For $a + b \ge 3$, we have:
    \begin{align*}
        \binom{a + b}{a} &= \binom{a + b - 1}{a} + \binom{a + b - 1}{a - 1} \\
        &\ge a(b - 1) + (a - 1)b + 2 \\
        &\ge a(b - 1) + (a - 1) + 2 = ab + 1.
    \end{align*}
\end{proof}

\begin{proposition}\label{prop:binom-convex-ineq}
    Let $a \ge 1$ and $b \ge 1$ be integers. Let $x$ be an integer with $a + 1 \le x \le a + b$. Then
    \[
        (a + b + 1 - x) \binom{x}{a} \ge ab + 1.
    \]
\end{proposition}
\begin{proof}
    For $x \in \R$, define $\binom{x}{a} = \frac{1}{a!} \cdot (x \cdot (x - 1) \cdots (x - a + 1))$.

    We prove the inequality for all real $x$ such that $a + 1 \le x \le a + b$.

    Let $f(x) = (a + b + 1 - x) \binom{x}{a}$.

    Compute the derivative of $f$. It is easy to see that
    \[
        \binom{x}{a}' = \binom{x}{a} \cdot \left( \sum_{i = 0}^{a - 1} \frac{1}{x - i} \right),
    \]
    hence,
    \[
        f'(x) = \binom{x}{a} \cdot \left( \sum_{i = 0}^{a - 1} \frac{a + b + 1 - x}{x - i} - 1 \right)
        = \binom{x}{a} \cdot \left( \sum_{i = 0}^{a - 1} \frac{a + b + 1 - i}{x - i} - a - 1 \right).
    \]

    Observe that on the interval $[a, \infty)$, the derivative is continuous and has exactly one zero. Also, $f'(a) > 0$ and $\lim_{x \to \infty} f'(x) < 0$. Therefore, to show that $f(x) \ge ab + 1$ on the interval $[a + 1, a + b] \subseteq [a, \infty)$, it suffices to check that the inequality holds at the boundary points. We have:
    \[
        f(a + 1) = b(a + 1) \ge ab + 1,
    \]
    and by Proposition~\ref{prop:binom-a-plus-b-ineq},
    \[
        f(a + b) = \binom{a + b}{a} \ge ab + 1.
    \]
\end{proof}

Our main bound on the size of the shadow is expressed in terms of the following quantity, which captures the relationship between the size of the shadow and the number of edges.
\begin{definition}\label{def:f-r-delta}
    Let $r \ge 1$ and $\delta \ge 1$ be integers, and let $G$ be an $r$-uniform hypergraph. Define
    \[
        f_{r,\delta}(G) = \delta \cdot \lvert \shadow{r-1}({G}) \rvert - r \cdot \lvert E(G) \rvert
        = \sum_{U \in \shadow{r-1}(G)} (\delta - \lvert L_G(U) \rvert).
    \]
\end{definition}


\begin{proposition}\label{prop:f-r-delta-shadow-bound}
    Let $r \ge 1$, $\delta \ge 2$, $e\ge 1$, $a \ge 1$ be integers, and let $G$ be an $r$-uniform hypergraph with $e$ edges. Suppose that $a \le (r-1)(\delta - 1) + 1$ and
    $e  \le \binom{r + \delta - 1}{r} - a$.
    Then
    \[
        f_{r,\delta}(G) \ge a.
    \]
\end{proposition}

\begin{proof}
    By Theorem~\ref{thrm:kruskal-katona}, we may assume $G = \lcG(r, e)$.

    We proceed by induction on $r$. For $r = 1$, we have
    \[
        f_{r,\delta}(G) = \delta \cdot \lvert \shadow{0}(G) \rvert - r \cdot e = \delta - e
        = \binom{r + \delta - 1}{r} - e \ge a.
    \]

    Suppose $r \ge 2$. Let $k$ be the largest integer such that $e \ge \binom{k}{r}$. Then $r \le k \le r + \delta - 2$.

    Let $m = e - \binom{k}{r}$. Then $0\le m \le \binom{k}{r - 1} - 1$.

    If $m = 0$, then $E(G) = \binom{[k]}{r}$, because $G$ is left-compressed. By Proposition~\ref{prop:binom-convex-ineq},
    \[
        f_{r,\delta}(G) = f_{r,\delta}\left( \binom{[k]}{r} \right) = (\delta + r - k - 1) \binom{k}{r - 1}
        \ge (r - 1)(\delta - 1) + 1 \ge a.
    \]

    Consider the case $m \ge 1$. Let $F = L_G(\{k + 1\})$. Since $G$ is left-compressed,
    \[
        \shadow{r - 1}(G) = \shadow{r - 1} \left( \binom{[k]}{r} \right) \cup
        \left\{ U \cup \{k + 1\} \,\middle|\, U \in \shadow{r - 2}(F) \right\},
    \]
    so
    \begin{equation}\label{eq:recursive-form-for-f-r-delta}
        f_{r,\delta}(G) = \delta \cdot \lvert \shadow{r - 1}(G) \rvert - r \cdot \lvert E(G) \rvert
        = f_{r,\delta}\left( \binom{[k]}{r} \right) + f_{r - 1, \delta}(F) - \lvert F \rvert.
    \end{equation}

    If $k \le r + \delta - 3$, then $\delta \ge 3$. By Proposition~\ref{prop:binom-a-plus-b-ineq},
    \[
        m \le \binom{r + \delta - 3}{r - 1}
        = \binom{r + \delta - 2}{r - 1} - \binom{r + \delta - 3}{r - 2}
        \le \binom{r + \delta - 2}{r - 1} - (r - 2)(\delta - 1) - 1.
    \]
    Then by the inductive hypothesis, $f_{r - 1, \delta}(F) \ge (r - 2)(\delta - 1) + 1$. Hence by \eqref{eq:recursive-form-for-f-r-delta} and Proposition~\ref{prop:binom-convex-ineq},
    \begin{align*}
        f_{r,\delta}(G)
        &\ge (\delta + r - k - 1) \binom{k}{r - 1} + (r - 2)(\delta - 1) + 1 - m \\
        &\ge (\delta + r - k - 2) \binom{k}{r - 1} + (r - 2)(\delta - 1) + 1 \\
        &\ge (r - 1)(\delta - 2) + 1 + (r - 2)(\delta - 1) + 1 \ge (r - 1)(\delta - 1) + 1 \ge a.
    \end{align*}

    Consider the case $k = r + \delta - 2$. By the inductive hypothesis, $f_{r - 1, \delta}(F) \ge 0$.

    Since
    \[
        \binom{r + \delta - 2}{r} + m = e \le \binom{r + \delta - 1}{r} - a,
    \]
    it follows that $m \le \binom{r + \delta - 2}{r - 1} - a$.

    Hence, by \eqref{eq:recursive-form-for-f-r-delta},
    \[
        f_{r,\delta}(G) = \binom{r + \delta - 2}{r - 1} + f_{r - 1, \delta}(F) - m
        \ge \binom{r + \delta - 2}{r - 1} - m \ge a.
    \]
\end{proof}

We also need the following proposition.

\begin{proposition}\label{prop:lower-bound-size-of-shadow-via-delta-star}
    Let $r \ge 1$ and $\delta \ge 1$ be integers. Let $H$ be a non-empty $r$-uniform hypergraph with $\delta^*(H) \ge \delta$.
    Then
    \[
        \lvert \shadow{r-1}(H) \rvert \ge \binom{r + \delta - 1}{r - 1}.
    \]
\end{proposition}

\begin{proof}
    Let $e = |E(H)|$. By Theorem~\ref{thrm:kruskal-katona}, it is sufficient to prove that
    \begin{equation}\label{eq:bound-on-e-is-sufficient:prop:lower-bound-size-of-shadow-via-delta-star}
        e \ge \binom{r + \delta - 1}{r},
    \end{equation}
    since then
    \[
        \lvert \shadow{r - 1}(H) \rvert
        \ge \lvert \shadow{r - 1}(\lcG(r, e)) \rvert
        \ge \left\lvert \shadow{r - 1} \left( \binom{[r + \delta - 1]}{r} \right) \right\rvert
        = \binom{r + \delta - 1}{r - 1}.
    \]

    Without loss of generality, we may assume that every vertex $v \in V(H)$ lies in some edge of $H$. 

    We now prove \eqref{eq:bound-on-e-is-sufficient:prop:lower-bound-size-of-shadow-via-delta-star} by induction on $r$. For $r = 1$, the inequality holds since
    \[
        e \ge \delta = \binom{r + \delta - 1}{r}.
    \]

    Suppose $r \ge 2$. Take an arbitrary $U \in \shadow{r - 1}(H)$. Since $|L_H(U)| \ge \delta$, it follows that
    \[
        |V(H)| \ge |U| + |\shadow{1}(L_H(U))| \ge r - 1 + \delta.
    \]

    Fix a vertex $v \in V(H)$. It is easy to show that $\delta^*(L_H(\{v\})) \ge \delta$. Then by the induction hypothesis,
    \[
        \lvert L_H(\{v\}) \rvert \ge \binom{r + \delta - 2}{r - 1}.
    \]
    Therefore,
    \[
        e = \frac{1}{r} \sum_{v \in V(H)} \lvert L_H(\{v\}) \rvert
        \ge \frac{|V(H)|}{r} \cdot \binom{r + \delta - 2}{r - 1}
        \ge \frac{r - 1 + \delta}{r} \cdot \binom{r + \delta - 2}{r - 1}
        = \binom{r + \delta - 1}{r}.
    \]
\end{proof}

The following corollary gives a slightly stronger form of the bound from Theorem~\ref{thrm:optimal-lower-bound-hyper}, and this form is more convenient to derive from Theorem~\ref{thrm:gamma-lower-bound-hyper}.

\begin{corollary}\label{collorary:optimal-lower-bound-simplification}
    Let $r \ge 1$ and $\delta \ge 1$ be integers. Let $H$ be a non-empty $r$-uniform hypergraph with $\delta^*(H) = \delta$. Then for all $n \ge \lvert V(H) \rvert$,
    \[
        \left(\frac{\delta}{r} - \frac{1}{\binom{r + \delta - 1}{r - 1}}\right) \cdot \left(\binom{n}{r - 1} - \lvert \shadow{r - 1}(H) \rvert\right) + |E(H)| - 1 \ge \left(\frac{\delta}{r} - \frac{1}{\binom{r + \delta - 1}{r - 1}}\right) \cdot \binom{n}{r - 1}.
    \]
\end{corollary}
\begin{proof}
    Since each element of $\shadow{r - 1}(H)$ lies in at least $\delta$ edges of $H$, we have
    \[
        r\cdot\lbr{E(H)} = \sum_{U \in \shadow{r - 1}(H)} \lvert L_H(U) \rvert \ge \delta \cdot \lvert \shadow{r - 1}(H) \rvert.
    \]

    By Proposition~\ref{prop:lower-bound-size-of-shadow-via-delta-star}, $\lvert \shadow{r - 1}(H) \rvert \ge \binom{r + \delta - 1}{r - 1}$, hence
    \[
        \left(\frac{\delta}{r} - \frac{1}{\binom{r + \delta - 1}{r - 1}}\right) \cdot \lvert \shadow{r - 1}(H) \rvert \le |E(H)| - 1.
    \]
\end{proof}

\subsection{Proof of Theorem~\ref{thrm:optimal-lower-bound-hyper}}

If $r = 1$, then the bound in Theorem~\ref{thrm:optimal-lower-bound-hyper} follows from Proposition~\ref{prop:wsat-behavior-r-1}.

If $\delta^*(H) = 1$, then the bound in the theorem is trivial.

For $\delta^*(H) \ge 2$, by Proposition~\ref{prop:sparseness-and-delta}, we have $s(H) = r$. Therefore, in view of Corollary~\ref{collorary:optimal-lower-bound-simplification}, to prove Theorem~\ref{thrm:optimal-lower-bound-hyper}, it suffices to establish the following lower bound on $\gamma_{r,H}$.

\begin{theorem}\label{thrm:lower-bound-on-gamma-via-delta}
    Let $r \ge 2$ and $\delta \ge 2$ be integers. Let $H$ be an $r$-uniform hypergraph with $\delta^*(H) = \delta$. Then
    \begin{equation}\label{eq:lower-bound-gamma}
        \gamma_{r,H} \ge \frac{\delta}{r} - \frac{1}{\binom{r + \delta - 1}{r - 1}}.
    \end{equation}
\end{theorem}

\begin{proof}
    Let $\mc{S}$ be a set at which the minimum in Proposition~\ref{prop:alt-form-gamma} is attained. Let $k = \lvert \mc{S} \rvert$, let
    \[
        G = \left\{ e \in E(H) \ \middle|\  \binom{e}{r - 1} \cap \mc{S} \neq \emptyset \right\},
    \]
    and let $e = \lvert E(G) \rvert$. If the bound in \eqref{eq:lower-bound-gamma} is violated, then
    \[
        \frac{e - 1}{k} < \frac{\delta}{r} - \frac{1}{\binom{r + \delta - 1}{r - 1}},
    \]
    which implies
    \begin{equation}\label{eq:upper-e}
        e < 1 + \frac{k\delta}{r} - \frac{k}{\binom{r + \delta - 1}{r - 1}}.
    \end{equation}

    Since each edge contains $r$ elements of $\shadow{r - 1}(G)$ and each element of $\mc{S}$ lies in at least $\delta$ edges of $H$, it follows that $r \cdot e \ge \delta \cdot k$. Combining with \eqref{eq:upper-e} gives
    \begin{equation}\label{eq:e-is-ceil}
        e = \left\lceil \frac{\delta k}{r} \right\rceil = \frac{\delta k + m}{r},
    \end{equation}
    where $m = r \cdot \left\lceil \frac{\delta k}{r} \right\rceil - \delta k$.

    Substituting the expression for $e$ from \eqref{eq:e-is-ceil} into \eqref{eq:upper-e} yields
    \begin{equation}\label{eq:upper-k}
        k < \left(1 - \frac{m}{r} \right) \binom{r + \delta - 1}{r - 1}.
    \end{equation}

    Substituting \eqref{eq:upper-k} into \eqref{eq:e-is-ceil} yields
    \[
        e < \left(1 - \frac{m}{r} \right) \binom{r + \delta - 1}{r} + \frac{m}{r} = \binom{r + \delta - 1}{r} - \frac{m}{r} \left( \binom{r + \delta - 1}{r} - 1 \right).
    \]

    By Proposition~\ref{prop:binom-a-plus-b-ineq}, $\binom{r + \delta - 1}{r} \ge r \cdot (\delta - 1) + 1$, so
    \begin{equation}\label{eq:bound-on-e-in-gamma-bound}
        e < \binom{r+\delta-1}{r}-\frac{m}{r}\br{r\cdot (\delta - 1)} = \binom{r+\delta-1}{r} - m(\delta - 1).
    \end{equation}

    For every $W \in \mc{S}$, we have $E(L_H(W)) = E(L_G(W))$, so
    \[
        r e = \sum_{W \in \shadow{r - 1}(G)} \lvert L_G(W) \rvert \ge \delta \cdot \lvert \mc{S} \rvert + \lvert \shadow{r - 1}(G) \setminus \mc{S} \rvert = \delta k + \lvert \shadow{r - 1}(G) \rvert - k.
    \]

    Since $r e = \delta k + m$, we have \[\lvert \shadow{r - 1}(G) \rvert \le k + m = \frac{r}{\delta} e - \frac{m}{\delta} + m.\] Hence, for the function $f_{r,\delta}$ from Definition~\ref{def:f-r-delta},
    \begin{equation}\label{eq:thrm:lower-bound-gamma:f-r-delta-bound}
        f_{r,\delta}(G) \le m (\delta - 1).
    \end{equation}

    Let $a = m (\delta - 1) + 1$. Then $a \le (r - 1)(\delta - 1) + 1$. By \eqref{eq:bound-on-e-in-gamma-bound}, $e \le \binom{r + \delta - 1}{r} - a$, so Proposition~\ref{prop:f-r-delta-shadow-bound} implies that $f_{r,\delta}(G) \ge a$, contradicting \eqref{eq:thrm:lower-bound-gamma:f-r-delta-bound}.
\end{proof}

\section{Optimality of the lower bound}\label{sec:optimality}

To prove upper bounds on weak saturation, we use the following theorem of R\"odl~\cite{1985Rodl}.

\begin{theorem}[\cite{1985Rodl}]\label{thrm:rodl}
    Let $t \ge 0$ and $k\ge t$ be integers, and let $\epsilon > 0$ be real. Then there exists $N_0(k, t, \epsilon) \ge k$ such that for all sets $X$ with $\lvert X \rvert \ge N_0(k,t,\epsilon)$, there exists a family $\mc{F}_X \subseteq \binom{X}{k}$ of size $\lvert \mc{F}_X \rvert \le (1+\epsilon)\cdot{\binom{\lvert X \rvert}{t}}/{\binom{k}{t}}$ such that every $A \in \binom{X}{t}$ is contained in some $W \in \mc{F}_X$.
\end{theorem}

The following proposition is the main tool for constructing a hypergraph $H$ with the desired bound on $\wsat(n,H)$.

\begin{proposition}\label{prop:upper-bound-via-G-hyper}
    Let $r \ge 1$, $\delta \ge 2$, and $s \ge 2$ be integers. Let $G$ be a non-empty $r$-uniform hypergraph such that $s(G) = s$, $\delta_{s-1}(G) = \delta$. Let $P \subseteq V(G)$ be a set of vertices such that $|P| \ge s-1$ and $\binom{P}{s-1}\seq\shadow{s-1}(G)$. Then there exists a non-empty $r$-uniform hypergraph $H$ such that $\delta_{s-1}(H) = \delta$, $s(H) = s$, and
    \[
        \wsat(n,H) \le \frac{\left\lvert \left\{ e \in E(G)\ \middle|\ |e \cap P| \ge s - 1 \right\} \right\rvert - 1}{\binom{|P|}{s-1}} \cdot \binom{n}{s-1} + o(n^{s-1}).
    \]
\end{proposition}

\begin{proof}
    Let $\wt{G} = \left\{ e \in E(G)\ \middle|\ |e \cap P| \ge s - 1 \right\}$.

    Let $W$ be a set of size $\lvert W \rvert = \lvert V(G) \rvert$ disjoint from $V(G)$. Consider the $r$-uniform hypergraph $H_0$ with vertex set $V(H_0) = V(G) \cup W$ and edge set
    \[
        E(H_0) = E(\wt{G}) \cup \left\{ e \in \binom{V(H_0)}{r}\ \middle|\ \lvert e \cap P \rvert \le s - 2 \right\}.
    \]

    Let $e_1, \ldots, e_m$ be the missing edges in $H_0$. Define $H_i = H_0 \cup \bigcup_{j=1}^{i} \{e_j\}$. Let $H = G \sqcup \bigsqcup_{i=0}^{m} H_i$. Since $\binom{P}{s-1}\seq \shadow{s-1}(G)\seq \shadow{s-1}(H_0)$, we have $\shadow{s-1}(H_0)=\binom{V(H_0)}{s-1}$ and $\delta_{s-1}(H_0)\ge\delta$. Therefore, $\delta_{s-1}(H_i) \ge \delta$ and $s(H_i) \ge s$ for all $i \ge 0$. Also, $s(G) = s$ and $\delta_{s-1}(G) = \delta$, so $s(H) = s$ and $\delta_{s-1}(H) = \delta$.

    By Proposition~\ref{prop:wsat-family-equal-wsat-disjoint}, it suffices to show that for the family $\mc{H} = \{G\} \cup \{H_i\}_{i=0}^m$, we have
    \[
        \wsat(n,\mc{H}) \le \frac{\left\lvert \left\{ e \in E(G)\ \middle|\ |e \cap P| \ge s - 1 \right\} \right\rvert - 1}{\binom{|P|}{s-1}} \cdot \binom{n}{s-1} + o(n^{s-1}).
    \]

    Fix $n \ge \lvert V(H_0) \rvert$. We will construct a weakly $\mc{H}$-saturated hypergraph on vertex set $[n]$.

    Fix a subset $Z \subseteq [n]$ of size $\lvert Z \rvert = \lvert V(H_0) \rvert - |P| \ge |V(G)|$. By Theorem~\ref{thrm:rodl}, there exists a family $\mc{F} \subseteq \binom{[n] \setminus Z}{\lvert P \rvert}$ such that every subset of $[n] \setminus Z$ of size $s-1$ is contained in some element of $\mc{F}$, and
    \[
        \lvert \mc{F} \rvert \le (1 + o(1)) \cdot \frac{\binom{n - \lvert Z \rvert}{s - 1}}{\binom{\lvert P \rvert}{s - 1}}.
    \]

    Take an arbitrary edge $\wt{e} \in E(\wt{G})$.

    The edge set of $F$ is constructed as follows. Include all edges $e$ such that $\lvert e \setminus Z \rvert \le s - 2$. Next, for each $X \in \mc{F}$, include in $F$ the edges of a copy of the hypergraph $H_0 \setminus \{\wt{e}\}$ on the vertex set $X \cup Z$, where $X$ serves as the preimage of the vertex set $P$. Then the number of edges in $F$ is bounded by
    \[
        \lvert E(F) \rvert \le \lvert \mc{F} \rvert \cdot (|E(\wt G)| - 1) + O(n^{s - 2}) \le \frac{(|E(\wt G)| - 1)}{\binom{\lvert P \rvert}{s - 1}} \cdot \binom{n}{s - 1} + o(n^{s-1}).
    \]

    To complete the proof of the theorem, it remains to show that $F$ is weakly $\mc{H}$-saturated.

    We add the missing edges in the following order. First, for each $X \in \mc{F}$, we add all missing edges inside $X \cup Z$. To do this, we first add the preimage of the edge $\wt{e}$, which creates a copy of $H_0$; then, the remaining edges on $X \cup Z$ can be added using the hypergraphs $\{H_i\}_{i=1}^m$ in increasing order of $i$. Doing this for each $X$ gives a hypergraph $\wt{F}$ that contains all edges $e$ such that $\lvert e \setminus Z \rvert \le s - 1$. To finish, we apply Proposition~\ref{prop:short-tuza-final-building-step} with $H = G$.
\end{proof}

We have the following corollary, which proves Theorem~\ref{thrm:hyper-lower-bound-is-optimal}.

\begin{corollary}\label{corollary:example-H-s-r}
    Let $r \ge 1$ and $\delta \ge 1$ be integers. Then there exists an $r$-uniform hypergraph $H$ with $\delta^*(H) = \delta$ such that
    \[
        \wsat(n,H) \le \left( \frac{\delta}{r} - \frac{1}{\binom{r+\delta - 1}{r - 1}} \right) \binom{n}{r-1} + o(n^{r-1}).
    \]
\end{corollary}

\begin{proof}
    In the case $r = 1$, one can take $H = K_{\delta}^1$. Thus, assume $r \ge 2$.

    In the case $\delta = 1$, one can take $H = K_r^r$. Thus, assume $\delta \ge 2$.

    The existence of the required $H$ follows from Proposition~\ref{prop:upper-bound-via-G-hyper} with $G = K_{r+\delta - 1}^{r}$ and $P = V(G)$.
\end{proof}

We also obtain the following result for arbitrary sparseness.

\begin{corollary}\label{corollary:example-H-s-arbitrary}
    Let $r,s,k$ be integers with $r\ge s\ge 2$ and $k \ge r + 1$. Let $\delta = \binom{k - s + 1}{r - s + 1}$. Then there exists an $r$-uniform hypergraph $H$ such that $s(H) = s$, $\delta_{s-1}(H) = \delta$, and
    \[
        \wsat(n,H) \le \left( \frac{\delta}{\binom{r}{s-1}} - \frac{1}{\binom{k}{s - 1}} \right) \binom{n}{s-1} + o(n^{s-1}).
    \]
\end{corollary}

\begin{proof}
    Take disjoint sets $A, B, C$ such that $|A| = k$, $|B| = s$, and $|C| = k$. Choose an arbitrary subset $W \subseteq C$ of size $r - s$. Let $G$ be an $r$-uniform hypergraph on the vertex set $V(G) = A \cup B \cup C$ with edge set
    \[
        E(G) = \binom{A}{r} \cup \{ B \cup W \} \cup \left\{ e \in \binom{V(G)}{r} \ \middle| \ |e \cap A| \le s - 2 \text{ and } |e \cap B| \le s - 1 \right\}.
    \]

    It is easy to verify that $s(G) = s$, $\shadow{s-1}(G) = \binom{V(G)}{s-1}$, and $\delta(G) = \delta$.

    The existence of the required $H$ follows from Proposition~\ref{prop:upper-bound-via-G-hyper} by taking $P = A$.
\end{proof}

\section{Optimal bounds for arbitrary sparseness}\label{sec:arbitrary-sharpness-bounds}

By Theorem~\ref{thrm:tuza-theta-wsat}, it is known that $\wsat(n,H) = \Theta(n^{s(H)-1})$. Therefore, in the case \mbox{$\delta^*(H) = 1$}, it is natural to seek a generalization of the lower bound in Theorem~\ref{thrm:optimal-lower-bound-hyper} that yields a bound of the form $\wsat(n,H) \ge \Omega(n^{s(H)-1})$. In view of Proposition~\ref{prop:lower-bound-via-delta-m}, the following question arises.

\begin{question}\label{qstn:optimal-lower-bound-delta-sharpness}
    Let $r \ge 1$ and $s \ge 2$ be integers. Find a function $f_s: \Nz \to \R$ such that for every $\delta \ge 2$ and every $r$-uniform hypergraph $H$ with $s(H) = s$ and $\delta_{s-1}(H) = \delta$, the following holds:
    \[
        \wsat(n,H) \ge f(\delta)\binom{n}{s-1} + o(n^{s-1}).
    \]
    In addition, for every integer $\delta \ge 2$ and every real $\epsilon > 0$, there should exist an $r$-uniform hypergraph $H$ with $s(H) = s$ and $\delta_{s-1}(H) = \delta$ such that
    \[
        \wsat(n,H) \le (f(\delta) + \epsilon)\binom{n}{s-1} + o(n^{s-1}).
    \]
\end{question}

By Theorem~\ref{thrm:tuza-limit-wsat}, such a function $f_s$ exists and is uniquely defined for $\delta \ge 2$.

We conjecture that using methods similar to those in Section~\ref{sec:gamma-lower-bound-hyper}, the following bound on $\gamma_{s,H}$ can be shown.

\begin{conjecture}
    Let $r \ge 2$, $s \ge 2$, and $k \ge r+1$ be integers. Let $\delta = \binom{k - s + 1}{r - s + 1}$. Let $H$ be an $r$-uniform hypergraph with $s(H) = s$ and $\delta_{s-1}(H) = \delta$. Then
    \[
        \gamma_{s,H} \ge \frac{\delta}{\binom{r}{s-1}} - \frac{1}{\binom{k}{s - 1}}.
    \]
\end{conjecture}

By the example in Corollary~\ref{corollary:example-H-s-arbitrary}, if this conjecture is true, it yields the values of $f_s(\delta)$ from Question~\ref{qstn:optimal-lower-bound-delta-sharpness} for $\delta$ of the form $\binom{k-s+1}{r-s+1}$.

For the lower bound in terms of $\delta_{s-1}(H)$ for $2 \le s \le r - 1$, unlike the bound in terms of $\delta_{r-1}(H)$, we were unable to find a natural upper bound analogous to Proposition~\ref{prop:trivial-upper-bound-hyper}. In this context, we introduce the quantity $\eta(H)$, for which there are both natural lower and upper bounds.

\begin{definition}
    Let $r \ge 1$ and $s \ge 1$ be integers. Let $H$ be a non-empty $r$-uniform hypergraph with $s(H) = s$. By the definition of sparseness, there exists $U \in \shadow{s-1}(H)$ such that $s(L_H(U)) = 1$. From Proposition~\ref{prop:sparseness-and-delta}, for every $U \in \shadow{s-1}(H)$, we have $|L_H(U)| \ge 2$. Therefore, for the family $\mc{H} = \{ L_H(U) \mid U \in \shadow{s-1}(H) \}$, Proposition~\ref{prop:wsat-behavior-s-1} applies. We define $\eta(H) = C_{\mc{H}}$.
\end{definition}

\begin{proposition}\label{prop:bounds-via-eta-H}
    Let $r \ge 1$ and $s \ge 2$ be integers. Let $H$ be a non-empty $r$-uniform hypergraph with $s(H) = s$. Then for all $n \ge |V(H)|$,
    \[
        \frac{\eta(H)}{\binom{r}{s-1}} \cdot \binom{n}{s-1} \le \wsat(n,H) \le \eta(H) \cdot \binom{n}{s-1} + O(n^{s-2}).
    \]
\end{proposition}

\begin{proof}
    Let $\mc{H} = \{ L_H(U) \mid U \in \shadow{s-1}(H) \}$.

    For the lower bound, take a hypergraph $F \in \Wsat(n,H)$. Fix an $U \in \binom{V(H)}{s-1}$. Then the hypergraph $L_F(U)$ belongs to $\Wsat(n - s + 1, \mc{H})$, since if we add edges in $L_F(U)$ in the order they appear in $F$, skipping those not containing $U$, then each time a copy of some $\tilde{H}$ is created in $F$, a copy of $L_{\tilde{H}}(U)$ is created in $L_F(U)$. Using the property $C_{\mc{H}}$ from Proposition~\ref{prop:wsat-behavior-s-1}, we get:
    \[
        \binom{r}{s-1} \cdot |E(F)| = \sum_{U \in \binom{V(H)}{s-1}} |L_F(U)| \ge \binom{n}{s-1} \cdot \wsat(n - s + 1, \mc{H}) \ge \binom{n}{s-1} \cdot \eta(H).
    \]

    Now we prove the upper bound. Let $m$ be the smallest integer such that $m \ge |V(H)|$ and $\wsat(m,\mc{H}) = \eta(H)$.

    Construct a weakly $H$-saturated hypergraph $F$ on vertex set $[n]$ for $n \ge m$. Fix a subset $Z \subseteq [n]$ with $|Z| = m$. Let $G$ be a weakly $\mc{H}$-saturated hypergraph on $Z$ with $|E(G)| = \eta(H)$.

    The edge set of $F$ is constructed as follows. Include all edges $e$ such that $|e \setminus Z| \le s - 2$. Then for each $U \in \binom{[n] \setminus Z}{s - 1}$, include in $F$ the edges $\{ U \cup e \mid e \in E(G) \}$. Then
    \[
        |E(F)| \le |E(G)| \cdot \binom{n}{s-1} + O(n^{s-2}) = \eta(H) \cdot \binom{n}{s-1} + O(n^{s-2}).
    \]

    It remains to show that $F$ is weakly $H$-saturated. We add edges in $F$ as follows. Fix $U \in \binom{V(H)}{s-1}$. Since we have all edges with $|e \setminus Z| \le s - 2$, we can add all edges $e$ such that $U \subseteq e$ and $e \setminus U \subseteq Z$ using the fact that $G \in \Wsat(m,\mc{H})$. After this, we obtain a hypergraph $\tilde{F}$ that contains all edges $e \in \binom{V(F)}{r}$ such that $|e \setminus Z| \le s - 1$. The remaining edges can be added using Proposition~\ref{prop:short-tuza-final-building-step}.
\end{proof}

By Proposition~\ref{prop:wsat-behavior-r-1}, for a hypergraph $H$ with $s(H) = r$, we have $\eta(H) = \delta^*(H) - 1$. Therefore, Proposition~\ref{prop:bounds-via-eta-H} generalizes Proposition~\ref{prop:trivial-upper-bound-hyper}.

For the quantity $\eta(H)$, we are also interested in determining asymptotically optimal lower and upper bounds.

\begin{question}\label{qstn:optimal-lower-bound-eta-sharpness}
    Let $r \ge 1$ and $s \ge 2$ be integers. Find functions $g_{\mathrm{lower},s}, g_{\mathrm{upper},s}: \Nz \to \R$ such that for every $\eta \ge 1$ and every $r$-uniform hypergraph $H$ with $s(H) = s$ and $\eta(H) = \eta$, the following holds:
    \[
        g_{\mathrm{lower},s}(\eta)\binom{n}{s-1} + o(n^{s-1}) \le \wsat(n,H) \le g_{\mathrm{upper},s}(\eta)\binom{n}{s-1} + o(n^{s-1}).
    \]
    In addition, for every integer $\eta \ge 1$ and every real $\epsilon > 0$, there should exist an $r$-uniform hypergraph $H$ with $s(H) = s$ and $\eta(H) = \eta$ such that
    \[
        \wsat(n,H) \le (g_{\mathrm{lower},s}(\eta) + \epsilon)\binom{n}{s-1} + o(n^{s-1}).
    \]
    Furthermore, for every integer $\eta \ge 1$, there should exist an $r$-uniform hypergraph $H$ with $s(H) = s$ and $\eta(H) = \eta$ such that
    \[
        \wsat(n,H) \ge g_{\mathrm{upper},s}(\eta)\binom{n}{s-1} + o(n^{s-1}).
    \]
\end{question}

\section{Asymptotics of polymatroidal method}\label{sec:polymatroid-asympt}


From Theorem~\ref{thrm:poly-lower-bound} it follows that $\wsat(n,H)\ge \rho_M(K_n^r)$ for every weakly $H$-saturated \mbox{$1$-polymatroid} on ground set $E(K_n^r)$. It is not hard to show that there exists a weakly $H$-saturated $1$-polymatroid that yields the best lower bound.

\begin{proposition}\label{prop:rhosat-exists}
    Let $r\ge 1$ and $n\ge 1$ be integers, and let $H$ be a non-empty $r$-uniform hypergraph. Then there exists a weakly $H$-saturated $1$-polymatroid $M=(E(K_n^r),\rho_M)$ such that for every weakly $H$-saturated $1$-polymatroid $N=(E(K_n^r),\rho_N)$
    \[\rho_M(K_n^r)\ge \rho_N(K_n^r).\]
\end{proposition}
\begin{proof}
    The conditions \eqref{eq:poly-boundness}, \eqref{eq:poly-monotone}, \eqref{eq:poly-submodular}, and \eqref{eq:condition-poly} impose linear constraints on the values of the function $\rho \colon \pwrs{E(K_n^r)} \to \R$. Hence the problem of maximizing $\rho(K_n^r)$ over all weakly \mbox{$H$-saturated} $1$-polymatroids is a linear program. Since the set of feasible solutions is non-empty and, by \eqref{eq:poly-boundness}, bounded, an optimal solution exists.
\end{proof}

By Proposition~\ref{prop:rhosat-exists}, the following quantity is well defined.

\begin{definition}
    Let $r\ge 1$ and $n\ge 1$ be integers and let $H$ be a non-empty $r$-uniform hypergraph. Denote by $\rhosat(n,H)$ the best lower bound on $\wsat(n,H)$ that can be obtained from Theorem~\ref{thrm:poly-lower-bound}.
\end{definition}

The asymptotic order of $\rhosat(n,H)$ equals that of $\wsat(n,H)$.
\begin{proposition}
    Let $r\ge 1$ be integer and let $H$ be a non-empty $r$-uniform hypergraph. Then
    \[\wsat(n,H)=\Theta(\rho(n,H)).\]
\end{proposition}
\begin{proof}
    Let $s=s(H)$.

    If $s=0$, then for all $n\ge 0$, $\wsat(n,H)=0=\rhosat(n,H)$, since we can take a $1$-polymatroid $M$ with rank function
    \[\forall F\seq K_n^r\quad \rho_M(F)=0.\]

    If $s=1$, then $|E(H)|\ge 2$, and by Theorem~\ref{thrm:tuza-theta-wsat} $\wsat(n,H)=\Theta(1)$. Also, $\rhosat(n,H)=\Theta(1)$, since we may take the $1$-polymatroid $M$ whose rank function is given by
    \[\forall F\seq K_n^r\quad \rho_M(F)=\begin{cases}1,&\quad \text{if }|E(F)|\neq 0;\\ 0,&\quad\text{otherwise}.\end{cases}\]
    This $1$-polymatroid is weakly $H$-saturated since $|E(H)|\ge 2$.

    If $s\ge 2$, then by Theorem~\ref{thrm:tuza-theta-wsat} $\wsat(n,H)=\Theta(n^{s-1})$. Also, $\rhosat(n,H)=\Theta(n^{s-1})$ by the proof of Theorem~\ref{thrm:gamma-lower-bound-hyper}.
\end{proof}

Moreover, by arguing analogously to the proof in~\cite{2025TerekhovTuza}, one can show that $\rhosat(n,H)$, like $\wsat(n,H)$, have an asymptotic coefficient.
\begin{theorem}\label{thrm:limit-of-rhosat-exists}
    Let $r\ge 1$ and $s\ge 1$ be integers. Let $H$ be an $r$-uniform hypergraph with $s(H)=s$. Then the following limit exists:
    \[\lim_{n\to+\infty}\frac{\rhosat(n,H)}{n^{s-1}}.\]
\end{theorem}
Due to the similarity of the proof, it is give in Appendix~\ref{appendix:rhosat}.

This theorem suggests that the following asymptotic equality may also hold.
\begin{conjecture}
    Let $r\ge 1$ and $s\ge 1$ be integers. Let $H$ be an $r$-uniform hypergraph with $s(H)=s$. Then
    \[\wsat(n,H)=(1+o(1))\rhosat(n,H).\]
\end{conjecture}

It is also natural to ask whether the polymatroidal method always gives the exact value. We do not expect this to be the case, but we have not found a counterexample.
\begin{question}
    Let $r\ge 1$ and $n\ge 1$ be integers and let $H$ be a non-empty $r$-uniform hypergraph. Is it always true that
    \[\wsat(n,H)=\rhosat(n,H)?\]
\end{question}



\section{Declaration of generative AI and AI-assisted technologies in the writing process}

During the preparation of this work the author used ChatGPT to improve the readability and language of the manuscript. After using this service, the author reviewed and edited the content as needed and takes full responsibility for the content of the published article.

    \printbibliography

\appendix
\section{Proof of Theorem~\ref{thrm:limit-of-rhosat-exists}}\label{appendix:rhosat}

Let $v = \lbr{V(H)}$. The statement of the theorem is equivalent to the existence of the limit
\[\lim_{n\to +\infty} \frac{\rhosat(n,H)}{\binom{n - v}{s-1}}.\]

Since $\rhosat(n,H)\le \wsat(n,H)$, from Theorem~\ref{thrm:tuza-theta-wsat} it follows that
\[C = \liminf_{n\to +\infty} \frac{\rhosat(n,H)}{\binom{n - v}{s-1}}\]
exists.

Fix $\epsilon > 0$. By the definition of $C$, there exists $m \geq v + s - 1$ such that
\[\rhosat(m,H)\le (C+\epsilon)\cdot \binom{m-v}{s-1}.\]

By Theorem~\ref{thrm:rodl}, there exists $N_0 \ge m - v$ such that for any set $X$ with $\lbr{X}\ge N_0$, there exists a family $\mc{F}_X \subseteq \binom{X}{m - v}$ such that every subset of $X$ of size $s - 1$ is contained in at least one element of $\mc{F}_X$, where
\[\lbr{\mc{F}_X} \le (1+\epsilon)\cdot\frac{\binom{\lbr{X}}{s-1}}{\binom{m-v}{s-1}}.\]
Furthermore, any subset of $X$ of size less than $s - 1$ is contained in some element of $\mc{F}_X$, as such a subset can always be extended to the size $s - 1$.

Fix $n \ge N_0 + v$. Let $M=(E(K^r_n),\rho_M)$ be a weakly $H$-saturated $1$-polymatroid with $\rhosat(n,H)=\rho_M(K^r_n)$.

For every complete hypergraph $G\seq K^r_n$ on $m$ vertices. The induced $1$-polymatroid $M|_G= (E(G),\rho_M|_{E(G)})$ is also weakly $H$-saturated, and, by definition of $\rhosat(m,H)$,
\begin{equation}\label{eq:induced-poly}
    \rho_M(G)\le \rhosat(m,H).
\end{equation}

We construct an $r$-uniform hypergraph $F$ on vertex set $[n]$ in the following way. Let $Z = [v]$ and $X = [n] \setminus Z$. For each element $W \in \mc{F}_X$, add to the hypergraph $F$ a clique on the vertex set $Z \cup W$. From \eqref{eq:induced-poly} we get that:
\begin{align*}
\rho_M(F)&\le \sum_{W\in \mc{F}_X}\rho_M\br{\binom{Z\cup W}{r}}\le |\mc{F}_X|\cdot \rhosat(m, H) \\
&\le(1+\epsilon)\cdot \frac{\binom{\lbr{X}}{s-1}}{\binom{m-v}{s-1}} \cdot (C+\epsilon)\cdot \binom{m-v}{s-1}=(1+\epsilon)(C+\epsilon)\cdot\binom{n-v}{s-1}.
\end{align*}
By the definition of $\mc{F}_X$, $F$ contains all edges $e \in \binom{[n]}{r}$ satisfying $\lbr{X\cap e} \le s-1$. By Proposition~\ref{prop:short-tuza-final-building-step}, $F$ is weakly $H$-saturated and so
\[\rhosat(n,H) = \rho_M(K^r_n)=\rho_M(F) \le (1+\epsilon)(C+\epsilon)\cdot\binom{n-v}{s-1}.\]
Given that $\epsilon>0$ can be chosen arbitrarily small, the theorem is proved.

\end{document}